\documentclass{amsart}
\usepackage{amssymb}
\usepackage{xcolor}

\usepackage{tikz}
\usetikzlibrary{matrix,cd}
\usepackage{spectralsequences}

\usepackage{xr-hyper}
\usepackage[pagebackref]{hyperref}
\hypersetup{%
  bookmarksnumbered=true,%
  bookmarks=true,%
  colorlinks=true,%
  linkcolor=blue,%
  citecolor=blue,%
  filecolor=blue,%
  menucolor=blue,%
  urlcolor=blue,%
  pdfnewwindow=true,%
  pdfstartview=FitBH}
  
\usepackage[nameinlink,capitalise,noabbrev]{cleveref}

\crefname{equation}{}{}

\usepackage[shortalphabetic]{amsrefs}

\newcommand{\F}{\mathbb{F}_2}

\newcommand{\QH}{\mathbb{H}}
\newcommand{\Q}{\mathbb{Q}}
\newcommand{\R}{\mathbb{R}}
\newcommand{\Z}{\mathbb{Z}}
\newcommand{\rH}{\mathrm{H}}
\newcommand{\ul}{\underline}
\newcommand{\ulF}{\ul{\F}}

\newcommand{\ulm}{\ul{m}}

\newcommand{\ulZ}{\ul{\Z}}
\newcommand{\ulf}{\ul{f}}
\newcommand{\ulg}{\ul{g}}
\newcommand{\ulw}{\ul{w}}
\newcommand{\mystery}{\ul{mgw}}
\newcommand{\mf}[1]{\ul{#1}}
\newcommand{\mpi}{\ul{\pi}}
\newcommand{\iso}{\cong}
\newcommand{\smsh}{\wedge}
\newcommand{\into}{\hookrightarrow}
\newcommand{\onto}{\twoheadrightarrow}
\newcommand{\rtarr}{\longrightarrow}
\newcommand{\ltarr}{\longleftarrow}
\newcommand{\xrtarr}[1]{\xrightarrow{#1}}
\newcommand{\SI}{\Sigma}
\newcommand{\Si}[1]{\Sigma^{#1}}

\newcommand{\Sp}{\mathbf{Sp}}
\newcommand{\Mack}{\mathbf{Mack}}
\newcommand{\phiZ}{\Psi}
\newcommand{\HphiZ}{\Psi}

\DeclareMathOperator{\Mod}{Mod}

\newif\ifDrawTikzPictures
\tikzstyle{encirc}=[draw,circle,inner sep=0.8pt,scale=0.7]

\newcommand{\MackKNoArrows}[5]
{\begin{tikzcd}[ampersand replacement=\&]
 \& #1 \& \\
 #2 \& #3 \& #4 \\
 \& #5 \& 
\end{tikzcd}}
\makeatletter
\pgfdeclareshape{phiZFshape}
{%
  \inheritsavedanchors[from=diamond]% this is nearly a diamond
  \inheritanchorborder[from=diamond]%
  \inheritanchor[from=diamond]{text}%
  \inheritanchor[from=diamond]{north}%
  \inheritanchor[from=diamond]{north west}%
  \inheritanchor[from=diamond]{north east}%
  \inheritanchor[from=diamond]{center}%
  \inheritanchor[from=diamond]{west}%
  \inheritanchor[from=diamond]{east}%
  \inheritanchor[from=diamond]{mid}%
  \inheritanchor[from=diamond]{mid west}%
  \inheritanchor[from=diamond]{mid east}%
  \inheritanchor[from=diamond]{base}%
  \inheritanchor[from=diamond]{base west}%
  \inheritanchor[from=diamond]{base east}%
  \anchor{south}{\pgf@process{\outernortheast}\pgf@x=0pt\pgf@y=-2\pgf@y}%
  \anchor{south east}{\pgf@process{\outernortheast}\pgf@y=-2\pgf@y}%
  \anchor{south west}{\pgf@process{\outernortheast}\pgf@x=-\pgf@x\pgf@y=-2\pgf@y}%
  \backgroundpath{
    \pgf@process{\outernortheast}%
    \pgf@xc=\pgf@x%
    \pgf@yc=\pgf@y%
    \pgfmathsetlength{\pgf@xa}{\pgfkeysvalueof{/pgf/outer xsep}}%
    \pgfmathsetlength{\pgf@ya}{\pgfkeysvalueof{/pgf/outer ysep}}%
    \advance\pgf@xc by-1.414213\pgf@xa%
    \advance\pgf@yc by-1.414213\pgf@ya%
    \pgfpathmoveto{\pgfqpoint{0pt}{-\pgf@yc}}%
    \pgfpathlineto{\pgfqpoint{\pgf@xc}{0pt}}%
    \pgfpathlineto{\pgfqpoint{0pt}{\pgf@yc}}%
    \pgfpathlineto{\pgfqpoint{-\pgf@xc}{0pt}}%
    \pgfpathlineto{\pgfqpoint{0pt}{-\pgf@yc}}%
    \pgfusepath{fill,stroke}%    
    \pgfpathmoveto{\pgfqpoint{0pt}{-\pgf@yc}}%
    \pgfpathlineto{\pgfqpoint{0pt}{-2\pgf@yc}}
    \pgfusepath{stroke}%    
}
\foregroundpath{
    \outernortheast%
    \pgf@xc=\pgf@x%
    \pgf@yc=\pgf@y%
    \pgf@yb=\pgf@yc%
    \pgfsetstrokecolor{gray}%
    \pgfsetfillcolor{gray}%
    \pgfpathcircle{\pgfqpoint{0pt}{-1.65\pgf@yc}}{0.25\pgf@yb}%
    \pgfusepath{fill,stroke}%    
  }%
}%

\pgfdeclareshape{phiZMshape}
{%
  \inheritsavedanchors[from=circle]% this is nearly a circle
  \inheritanchorborder[from=circle]%
  \inheritanchor[from=circle]{north}%
  \inheritanchor[from=circle]{north west}%
  \inheritanchor[from=circle]{north east}%
  \inheritanchor[from=circle]{center}%
  \inheritanchor[from=circle]{west}%
  \inheritanchor[from=circle]{east}%
  \inheritanchor[from=circle]{mid}%
  \inheritanchor[from=circle]{mid west}%
  \inheritanchor[from=circle]{mid east}%
  \inheritanchor[from=circle]{base}%
  \inheritanchor[from=circle]{base west}%
  \inheritanchor[from=circle]{base east}%
  \inheritanchor[from=circle]{south}%
  \anchor{south}{\centerpoint
        \pgf@xa=\radius
        \advance\pgf@y by -2.6\pgf@xa
        }
  \anchor{south west}{\pgfpointorigin \pgf@x=-\radius \pgf@y= -\radius \multiply \pgf@y by 3}
  \anchor{south east}{\pgfpointorigin \pgf@x=\radius \pgf@y= -\radius \multiply \pgf@y by 3}
  \inheritbackgroundpath[from=circle]%
  \foregroundpath{
    \centerpoint%
    \pgf@xc=\pgf@x%
    \pgf@yc=\pgf@y%
    \pgfutil@tempdima=\radius%
    \pgfmathsetlength{\pgf@xb}{\pgfkeysvalueof{/pgf/outer xsep}}%
    \pgfmathsetlength{\pgf@yb}{\pgfkeysvalueof{/pgf/outer ysep}}%
    \ifdim\pgf@xb<\pgf@yb%
      \advance\pgfutil@tempdima by-\pgf@yb%
    \else%
      \advance\pgfutil@tempdima by-\pgf@xb%
    \fi%
    \pgfpathmoveto{\pgfpointadd{\pgfqpoint{\pgf@xc}{\pgf@yc}}{\pgfqpoint{0\pgfutil@tempdima}{-1\pgfutil@tempdima}}}%
    \pgfpathlineto{\pgfpointadd{\pgfqpoint{\pgf@xc}{\pgf@yc}}{\pgfqpoint{0\pgfutil@tempdima}{-1.6\pgfutil@tempdima}}}%
    \pgfusepath{fill,stroke}%    
    \pgfpathcircle{\pgfpointadd{\pgfqpoint{\pgf@xc}{\pgf@yc}}{\pgfqpoint{0pt}{-2.1\pgfutil@tempdima}}}{0.4\pgfutil@tempdima}%
    \pgfusepath{stroke}%    
  }%
}%

\makeatother

\def\dbox{\dot{\Box}}

\def\abox{
     \begin{tikzpicture}
        \node at (0,0) [rectangle,draw,fill={gray!50},inner sep=0.75pt] {$\ast$};
     \end{tikzpicture}
}
\def\dbox{
     \begin{tikzpicture}
        \node at (0,0) [rectangle,draw,inner sep=0.75pt] {$\ast$};
     \end{tikzpicture}
}

\def\atrap{
    \begin{tikzpicture}
      \node[trapezium, fill={gray!50}, draw, inner sep=2.5pt,scale=1] at (0,0) {};
      \node at (0,0) {$\ast$};
    \end{tikzpicture}
}

\def\apent{
    \begin{tikzpicture}
      \node[regular polygon, fill={gray!50}, draw, regular polygon sides=5, 
 minimum width=0pt, inner sep = 0.5ex,] at (0,0) {};
      \node at (0,0) {$\ast$};
    \end{tikzpicture}
}
\def\mysterysymbol{
    \begin{tikzpicture}
      \node[fill = black, regular polygon, regular polygon sides=5,  regular polygon rotate=180, minimum width=0pt,  inner sep = 0.3ex,scale=1.7] at (0,0) {};
    \end{tikzpicture}
}
\def\phiZF{
    \begin{tikzpicture}[x=1.0ex,y=1.0ex]
      \draw (0,0.3) circle (0.3ex);
      \draw (0,0.6) -- (0,1.0);
      \begin{scope}[yshift={-0.1ex}]
      \fill (0,1) -- (1,2) -- (0,3) -- (-1,2) -- cycle;
      \end{scope}
    \end{tikzpicture}
}
\def\phiZFstar{
    \begin{tikzpicture}[x=1.0ex,y=1.0ex]
      \draw (0,0.4) circle (0.3ex);
      \draw (0,0.7) -- (0,1.1);
      \filldraw[fill={gray!50}] (0,1) -- (1,2) -- (0,3) -- (-1,2) -- cycle;
      \node at (0,2) {$\ast$};
    \end{tikzpicture}
}

\def\phiZM{
    \begin{tikzpicture}[x=1.0ex,y=1.0ex]
      \node[phiZMshape,draw,inner sep=0.4ex] at (0,0) {};
    \end{tikzpicture}
}

\def\btrap{
\begin{tikzpicture}
  \node[trapezium, fill=black, draw, inner sep=2.5pt,scale=1] at (0,0) {};
\end{tikzpicture}
}
\def\gcirc{
\begin{tikzpicture}[x=1.2ex,y=1.2ex]
    \draw (0,-0.65) circle (5pt);
    \node at (0,-0.65) {\(n\)};
\end{tikzpicture}
}
\def\bpent{
\begin{tikzpicture}
  \node[regular polygon, fill=black, draw, regular polygon sides=5, 
minimum width=0pt, inner sep = 0.5ex,] at (0,0) {};
\end{tikzpicture}
}

\newcommand{\onecolor}{blue}
\newcommand{\twocolor}{orange}

\newtheorem{thm}{Theorem}[section]
\newtheorem*{thm*}{Theorem}
\newtheorem{cor}[thm]{Corollary}
\newtheorem{prop}[thm]{Proposition}

\theoremstyle{definition}
\newtheorem{defn}[thm]{Definition}

\newtheorem{eg}[thm]{Example}

\newtheorem{rmk}[thm]{Remark}

\makeatletter
\let\c@equation\c@thm
\makeatother

\numberwithin{equation}{section}

\newenvironment{pf}{\begin{proof}}{\end{proof}}

\title{The  slices of quaternionic Eilenberg-Mac~Lane spectra}

\author{Bertrand J. Guillou }
\address{Department of Mathematics, The University of Kentucky, Lexington, KY 40506--0027}
\email{bertguillou@uky.edu}
\author{Carissa Slone}
\address{Department of Mathematics, The University of Kentucky, Lexington, KY 40506--0027}
\email{c.slone@uky.edu}
\date{\today}

\thanks{The authors were  supported by NSF grant DMS-2003204.}

\begin{document}

\begin{abstract}
    We compute the slices and slice spectral sequence of integral suspensions of the equivariant Eilenberg-Mac~Lane spectra $H\ulZ$ 
    for the group of equivariance $Q_8$. Along the way, we compute the Mackey functors $\mpi_{k\rho} H\ulZ$.
\end{abstract}

\maketitle

\tableofcontents

\section{Introduction}

Let $G$ be a finite group.
The $G$-equivariant slice filtration  was first defined in the context of $G$-equivariant stable homotopy theory by Dugger in \cite{Dug}; it came to prominence as a result of its role in the proof of the Kervaire invariant conjecture by Hill, Hopkins, and Ravenel  \cite{Kervaire}. The slice filtration is an analogue in the $G$-equivariant stable homotopy category of the classical Postnikov filtration of spectra. One can also define a $G$-equivariant Postnikov filtration; on passage to fixed points with respect to any subgroup $H\leq G$, this recovers the Postnikov filtration of the $H$-fixed point spectrum. However, there are many equivariant spectra which possess a periodicity with respect to suspension by a $G$-representation sphere, and this periodicity is not visible in the $G$-equivariant Postnikov filtration. The slice filtration was devised by Dugger in order to display this periodicity for the case of the $C_2$-spectrum $K\R$.

Since the groundbreaking work \cite{Kervaire}, a number of authors have calculated the slice filtration, as well as the associated slice spectral sequence, for $G$-spectra of interest. 
A few cases are understood for an arbitrary finite group $G$. If $\ul{M}$ is a $G$-Mackey functor, then the equivariant Eilenberg-Mac~Lane spectrum $H_G \ul{M}$ is always a 0-slice \cite{Kervaire} (in this article, we use the ``regular'' slice filtration, as introduced in \cite{Ullman}). The slice filtrations of $\Si1 H_G\ul{M}$ and $\Si{-1} H_G \ul{M}$ were described in \cite{Ullman}.
The slices of certain suspensions of equivariant Eilenberg-Mac~Lane spectra were determined for $G$ an odd cyclic $p$-group in \cite{HHR2}, \cite{Yarnall} and \cite{V}, for dihedral groups of order $2p$, where $p$ is odd, in \cite{Zou}, and for the Klein-four group in \cite{GY} and \cite{Slone}. We extend this list by considering in this article the case of $G=Q_8$.

Some of the most far-reaching applications of the slice filtration and associated spectral sequence have come in the case of cyclic $p$-groups of equivariance. In addition to \cite{Kervaire}, this  also includes \cite{HHR}, \cite{MSZ}, \cite{Sulyma}, and \cite{HSWX}.
In particular, in \cite{HSWX} the authors use slice technology to understand a $C_4$-equivariant, height 4 Lubin-Tate theory at the prime 2. For each height $n$, there is a height $n$ Lubin-Tate theory that comes equipped with an action of the height $n$ (profinite) Morava stabilizer group. The homotopy fixed points with respect to this action gives a model for the $K(n)$-local sphere, a central object of study.
More approachable are the homotopy fixed points with respect to finite subgroups. At height 4, the Morava stabilizer group contains a $C_4$-subgroup (in fact a $C_8$), which gives the context for \cite{HSWX}. On the other hand, at height $2m$, where $m$ is odd, the Morava stabilizer group contains a $Q_8$-subgroup. Therefore it is possible that $Q_8$-equivariant slice techniques will eventually shed light on the $K(n)$-local sphere when $n=2m$ and $m$ is odd.

\medskip

The focus of our article is the determination of the slices of $\Si{n} H_{Q_8} \ulZ$. We list the slices in \cref{sec:Q8slice} and describe the associated spectral sequence in \cref{sec:SpectralSequences}.
We rely heavily on the computation of the slices of $\Si{n}H_{K_4} \ulZ$ given by the second author in \cite{Slone}. The quotient map $Q_8 \rtarr K_4$ allows us to gain insight into the $Q_8$-equivariant slices from the $K_4$-case, as we now explain in greater generality.

Given a normal subgroup $N \unlhd G$, there are several constructions that will produce a $G$-spectrum from a $G/N$-spectrum. First is the ordinary pullback, or inflation, functor. If $q\colon G \rtarr G/N$ is the quotient, then inflation is denoted $q^* \colon \Sp^{G/N} \rtarr \Sp^G$; it is left adjoint to the $N$-fixed point functor. This inflation functor plays an important role. For instance $q^*( S^0_{G/N})$ is equivalent to $S^0_G$. However, from our point of view, this construction has two deficiencies. First, the ordinary inflation does not interact well with the slice filtration. Secondly, the inflation of an $H_{G/N}\ulZ$-module does not have a canonical $H_G \ulZ$-module structure. 

On the other hand, the ``geometric inflation'' functor (\cite{SlicePrimer}*{Definition~4.1}, \cite{LMS}*{Section II.9})
\[
\phi_N^* \colon \Sp^{G/N} \rtarr \Sp^G,
\]
which is right adjoint to the geometric fixed points functor, interacts well with slices. Namely, if $N$ is a normal subgroup of order $d$ and $X$ is a $G/N$-spectrum, then 
\[ \phi_N^* P^k_k (X) \simeq P^{dk}_{dk} \left( \phi_N^* X \right), 
\]
by \cite{Ullman}*{Corollary~4-5} (see also \cite{SlicePrimer}*{Section~4.2}). However, in general the geometric inflation of an $H_{G/N}\ulZ$-module will not be an $H_G \ulZ$-module. 

The third variant is the $\ulZ$-module inflation functor 
(\cite{Z}*{Section~3.2})
\[
\HphiZ_N^*\colon \Mod_{H_{G/N} \ulZ} \rtarr \Mod_{H_G \ulZ}.
\]
By design, the $\ulZ$-module inflation of an $H_{G/N}\ulZ$-module has a canonical $H_G\ulZ$-module structure, though in general this functor does not interact well with the slice filtration.

In some cases, these constructions agree. For instance, if the underlying spectrum of the $G/N$-spectrum $X$ is contractible, then $q^* X \simeq \phi_N^* X$. If $X$ is furthermore an $H_{G/N}\ulZ$-module, then the three inflation functors coincide on $X$ (\cref{HphiZgeometric}).

The above discussion applies to the slices of $\Si{n}H_{G/N}\ulZ$: all slices, except for the bottom slice, have trivial underlying spectrum. It follows that these inflate to give many of the slices of $\Si{n}H_G\ulZ$.

Our main result along these lines, \cref{MainSliceInflationTheorem}, describes the higher slices of such an inflated $H_{G}\ulZ$-module. In the case of $G=Q_8$, $N=Z(Q_8)$, and $G/N=Q_8/Z \iso K_4$, it gives the following:

\begin{thm}
\label{IntroSliceInflationTheorem}
Let $n\geq 0$.
Then the nontrivial slices of $\Si{n} H_{Q_8} \ulZ$, above level $2n$, are
\[ P^{2k}_{2k} \left( \Si{n}H_{Q_8} \ulZ \right) \simeq \HphiZ_Z^* P^k_k \left( \Si{n} H_{K_4} \ulZ \right)
\simeq \phi_Z^* P^k_k \left( \Si{n} H_{K_4} \ulZ \right)
\]
for $k>n$. Furthermore,
\[
P^{2k}_n \left( \Si{n}H_{Q_8} \ulZ \right) \simeq \HphiZ_Z^* P^k_n \left( \Si{n} H_{K_4} \ulZ \right).
\]
\end{thm}

As the slices of $\Si{n} H_{K_4} \ulZ$ were determined by the second author in \cite{Slone}, this immediately provides all of the slices of $\Si{n} H_{Q_8} \ulZ$ above level $2n$.
The remaining slices of $\Si{n}H_Q \ulZ$ are then given by analyzing the slice tower of $\HphiZ_N^* (P^n_n H_K \ulZ)$. We perform this analysis in \cref{sec:phiZK4Zslicetowers}.

\subsection{Notation}
\label{sec:notn}

Throughout, whenever referencing the slice filtration, we will always mean the ``regular'' slice filtration of \cite{Ullman}.

We will often write simply $Q$ and $K$ to denote the quaternion group $Q_8$ and Klein four group $K_4$, respectively.
We write $Z$ for the central subgroup of $Q$ of order two generated by $z=-1$.
We write 
\[ 
L = \langle i \rangle, \qquad 
D = \langle k \rangle, \qquad \text{and} \qquad
R = \langle j \rangle \qquad 
\]
for the normal, cyclic subgroups of $Q$ of order 4. 
We also use the same names for the images of these subgroups in $Q/Z \iso K$.
In other words, the subgroup lattices of $Q_8$ and $K_4$ are
\[
\begin{tikzpicture}
\node (G) at (0,3) {$Q_8$};
\node (I) at (-1.2,2) {$L$};
\node (J) at (0,2) {$D$};
\node (K) at (1.2,2) {$R$};
\node (Z) at (0,1) {$Z$};
\node (e) at (0,0) {$e$};
\begin{scope}[xshift=.4\textwidth]
\node (Klein) at (0,3) {$K_4$};
\node (L) at (-1.2,2) {$L$};
\node (D) at (0,2) {$D$};
\node (R) at (1.2,2) {$R$};
\node (eK) at (0,1) {$e$};
\end{scope}

\draw (G) to (I);
\draw (G) to (J);
\draw (G) to (K);
\draw (I) to (Z);
\draw (J) to (Z);
\draw (K) to (Z);
\draw (Z) to (e);

\draw (Klein) to (L);
\draw (Klein) to (D);
\draw (Klein) to (R);
\draw (eK) to (L);
\draw (eK) to (D);
\draw (eK) to (R);
\end{tikzpicture}
\]
Our nomenclature for the order 4 subgroups of $Q_8$ amounts to a choice of  isomorphism $Q/Z \iso K$.

The sign representation of $C_2$ will be denoted $\sigma$, and we will write $\Z^\sigma$ for the corresponding $C_2$-module. 

\subsection{Organization}

The paper is organized as follows. In \cref{sec:background}, we review the representations of \(C_4\), \(K_4\), and \(Q_8\), as well as Mackey functors over \(C_4\) and \(K_4\). Then in \cref{Inflationsection}, we introduce three inflation functors from a quotient group \(G/N\) of some finite group \(G\) as well as several results that will aid in the calculation of the slices of \(\Si{n} H_{Q_8} H\ulZ\). The relevant \(Q_8\)-Mackey functors and the homology of \(\Si{ k\rho_{Q_8}} H_{Q_8}\ulZ\) are found in \cref{Qsection}. The slices of \(\Si{n} H_{Q_8}\ulZ\) must restrict to the appropriate slices of \(\Si{n} H_{C_4}\ulZ\); thus, we review this information in \cref{sec:C4slice}. We provide some slice towers and describe all slices of \(\Si{n} H_{Q_8}\ulZ\) in \cref{sec:Q8slice}. We then compute the homotopy Mackey functors of the slices of \(\Si{n} H_{Q_8}\ulZ\) in \cref{sec:Homology}. Finally, we provide some examples of the slice spectral sequence for \(\Si{n} H_{C_4}\ulZ\) and \(\Si{n} H_{Q_8}\ulZ\) in \cref{sec:SpectralSequences}.

\subsection{Acknowledgements}

The authors are very happy to thank Agnes Beaudry, Michael Geline, Cherry Ng, and Mincong Zeng for a number of helpful discussions.
The spectral sequence charts in \cref{sec:SpectralSequences} were created using Hood Chatham's {\tt spectralsequences} package.

\section{
Background} \label{sec:background}

\subsection{Background for $C_4$}
\label{backgroundC4}

The $C_4$-sign representation $\sigma_{C_4}$ is the inflation $p^* \sigma_{C_2}$ of the $C_2$-sign representation along the surjection $C_4 \rtarr C_2$. We will simply write $\sigma$ for $\sigma_{C_4}$. Then the regular representation for $C_4$ splits as
\[ \rho_{C_4} = 1 \oplus \sigma \oplus \lambda,\]
where $\lambda$ is the irreducible 2-dimensional rotation representation of $C_4$.
The $RO(C_4)$-graded homotopy Mackey functors of $H_{C_4} \ul{\Z}$ are given in \cite{HHR}.
More specifically, the homotopy Mackey functors of $\Si{k\rho_{C_4}} H_{C_4} \ulZ$, $\Si{k\lambda} H_{C_4} \ulZ$, and $\Si{k\sigma} H_{C_4} \ulZ$ are given in Figures 3 and 6 of \cite{HHR}.
Some $C_4$-Mackey functors that will appear below are displayed in \cref{tab-C4Mackey}.
All of these Mackey functors have trivial Weyl-group actions.

\begin{table}%[h] 
\caption[Some  {$C_{4}$}-Mackey functors]{Some $C_{4}$-Mackey functors}
\label{tab-C4Mackey}
\begin{center}
{\renewcommand{\arraystretch}{1.2}
\begin{tabular}{|c|c|c|c|}
\hline 
$\square=\ulZ$
   & $\dbox=\ulZ^*$ 
   & $\ulZ(2,1)$
         &  $\circ=\ul B(2,0) $
         \\
         \hline
         \begin{tikzcd}[bend right, swap]
         \Z \ar[d,"1", color=\onecolor]  \\
         \Z \ar[d,"1", color=\onecolor] \ar[u,"2", color=\twocolor] \\
         \Z  \ar[u,"2", color=\twocolor]
         \end{tikzcd}
         & 
        \begin{tikzcd}[bend right,swap]
         \Z \ar[d,"2", color=\twocolor]  \\
         \Z \ar[d,"2", color=\twocolor] \ar[u,"1", color=\onecolor] \\
         \Z  \ar[u,"1", color=\onecolor]
         \end{tikzcd}
         & 
        \begin{tikzcd}[bend right, swap]
         \Z \ar[d,"2", color=\twocolor]  \\
         \Z \ar[d,"1", color=\onecolor] \ar[u,"1", color=\onecolor] \\
         \Z  \ar[u,"2", color=\twocolor]
         \end{tikzcd}
         & 
         \begin{tikzcd}[bend right, swap]
         \Z/4 \ar[d,"1"]  \\
         \Z/2 \ar[u,"2"] \\
         0  
         \end{tikzcd}
         \\
         \hline
         $\bullet=\ulg$
         &  $\overline{\bullet}=\phi^*\ulf$ 
         & $\raisebox{-0.35ex}{$\phiZF$}=\phi^*\ulF$
         & $\phi^*\ulF^*$
         \\ \hline
          \begin{tikzcd}
            \F \\ 0 \\ 0  
          \end{tikzcd}
         &  
         \begin{tikzcd}
            0 \\ \F \\ 0  
          \end{tikzcd}
         &    
         \begin{tikzcd}
            \F \ar[d,"1"] \\ \F \\ 0  
          \end{tikzcd}
         & 
          \begin{tikzcd}
            \F  \\ \F \ar[u,"1"] \\ 0  
          \end{tikzcd}
\\
\hline
\end{tabular} }
\end{center}
\end{table}

\subsection{Background for $K_4$}
\label{backgroundK4}

The Klein 4-group $K_4=C_2\times C_2$ has three sign representations, obtained as the inflation along the three surjections $K_4 \rtarr C_2$. We denote these three surjections by $p_1$, $m$, and $p_2$. Then the regular representation of $K_4$ splits as
\[ \rho_{K_4} \iso 1 \oplus p_1^* \sigma \oplus m^* \sigma \oplus p_2^* \sigma.\]
Some $K_4$-Mackey functors that will appear below  are displayed in \cref{tab-K4Mackey}.

\begin{table}%[h] 
\caption{Some $K_{4}$-Mackey functors}
\label{tab-K4Mackey}
\begin{center}
{\renewcommand{\arraystretch}{1.5}
\begin{tabular}{|c|c|c|}
\hline 
$\square = \ulZ$ &
$\dbox = \ulZ^*$ & 
         $\ulZ(2,1)$ \\
         \hline
         \begin{tikzcd}[bend right=2.5ex, swap]
         & \Z \ar[dl,"1", color=blue] \ar[d,"1", color=blue] \ar[dr,"1", color=blue] & \\
         \Z \ar[dr,"1", color=blue] \ar[ur,"2", color=orange]  & \Z \ar[d,"1", color=blue] \ar[u,"2", color=orange]  & \Z \ar[ul,"2", color=orange]  \ar[dl,"1", color=blue] \\
          & \Z  \ar[ul,"2", color=orange] \ar[u,"2", color=orange] \ar[ur,"2", color=orange] 
         \end{tikzcd}
         & 
         \begin{tikzcd}[bend right=2.5ex, swap]
         & \Z \ar[dl,"2", color=orange] \ar[d,"2", color=orange] \ar[dr,"2", color=orange] & \\
         \Z \ar[dr,"2", color=orange] \ar[ur,"1", color=blue]  & \Z \ar[d,"2", color=orange] \ar[u,"1", color=blue]  & \Z \ar[ul,"1", color=blue]  \ar[dl,"2", color=orange] \\
          & \Z \ar[ul,"1", color=blue] \ar[u,"1", color=blue] \ar[ur,"1", color=blue] 
         \end{tikzcd}
         &
         \begin{tikzcd}[bend right=2.5ex, swap]
         & \Z \ar[dl,"2", color=orange] \ar[d,"2", color=orange] \ar[dr,"2", color=orange] & \\
         \Z \ar[dr,"1", color=blue] \ar[ur,"1", color=blue]  & \Z \ar[d,"1", color=blue] \ar[u,"1", color=blue]  & \Z \ar[ul,"1", color=blue]  \ar[dl,"1", color=blue] \\
          & \Z  \ar[ul,"2", color=orange] \ar[u,"2", color=orange] \ar[ur,"2", color=orange] 
         \end{tikzcd}
         \\
         \hline
        $\blacksquare=\ulF$ 
        &
         $\abox=\ulF^*$
         &
         $\circ=\ul B(2,0) $
         \\
         \hline
         \begin{tikzcd}[ swap]
         & \F \ar[dl,"1"] \ar[d,"1"] \ar[dr,"1" swap]  & \\
         \F \ar[dr,"1" pos=0.3] & \F \ar[d,"1"]  & \F \ar[dl,"1" swap, pos=0.4] \\
          & \F  
         \end{tikzcd}
         &
         \begin{tikzcd}[swap]
         & \F  & \\
         \F  \ar[ur,"1" swap]  & \F \ar[u,"1"]  & \F \ar[ul,"1"]  \\
          & \F  \ar[ul,"1" swap] \ar[u,"1"] \ar[ur,"1"] 
          \end{tikzcd}
         &
         \begin{tikzcd}[bend right=2.5ex, swap]
         & \Z/4 \ar[dl,"1"] \ar[d,"1"] \ar[dr,"1"] & \\
         \Z/2 \ar[ur,"2"]  & \Z/2 \ar[u,"2"]  & \Z/2 \ar[ul,"2"]  \\
          & 0 
         \end{tikzcd}
         \\ 
         \hline 
         $\begin{tikzpicture}
      \node[regular polygon, fill=black, draw, regular polygon sides=5, 
 minimum width=0pt, inner sep = 0.5ex,] at (0,0) {};
    \end{tikzpicture} =
\phi^*_{LDR} (\ulF)$
         &
         $\raisebox{-3pt}{\apent} = \phi^*_{LDR}(\ulF)^*$
         &
         $\phi^*_{LDR}(\ul f)$
         \\ \hline
         \begin{tikzcd}[ swap]
         & \F^3 \ar[dl,"p_1"] \ar[d,"p_2"] \ar[dr,"p_3" swap]  & \\
         \F  & \F  & \F \\
          & 0  
         \end{tikzcd}
         &
         \begin{tikzcd}[]
         & \F^3    & \\
         \F \ar[ur,"\iota_1"]  & \F \ar[u,"\iota_2"'] & \F \ar[ul,"\iota_3" swap] \\
          & 0  
         \end{tikzcd}
         &
         \begin{tikzcd}[ swap]
         & 0 & \\
         \F  & \F  & \F \\
          & 0  
         \end{tikzcd}
         \\ 
         \hline 
         $\begin{tikzpicture}
      \node[trapezium, fill=black, draw, 
 minimum width=0pt, inner sep = 0.5ex,] at (0,0) {};
    \end{tikzpicture} = \ul{mg}$
         &
         $\raisebox{-3pt}{\begin{tikzpicture}
      \node[trapezium, fill={gray!50}, draw, 
 minimum width=0pt, inner sep = 0.5ex,] at (0,0) {};
 \node at (0,0) {$\ast$};
    \end{tikzpicture}} = \ul{mg}^*$
         &
         $\bullet=\ulg$
         \\ \hline
         \begin{tikzcd}[ swap]
         & \F^2 \ar[dl,"p_1"] \ar[d,"\nabla"] \ar[dr,"p_2" swap]  & \\
         \F  & \F  & \F \\
          & 0  
         \end{tikzcd}
         &
         \begin{tikzcd}[]
         & \F^2     & \\
         \F\ar[ur, "\iota_1"]  & \F \ar[u,"\Delta"'] & \F \ar[ul,"\iota_2" swap] \\
          & 0  
         \end{tikzcd}
         &
         $\MackKNoArrows{\F}0000$
         \\ 
         \hline 
         $\ulm$
         &
         $\ulm^*$
         &
         \\ \hline
         \begin{tikzcd}[ swap]
         & \F \ar[dl,"1"] \ar[d,"1"] \ar[dr,"1" swap]  & \\
         \F  & \F  & \F \\
          & 0  
         \end{tikzcd}
         &
         \begin{tikzcd}[]
         & \F    & \\
         \F \ar[ur,"1"]  & \F \ar[u,"1"'] & \F \ar[ul,"1" swap] \\
          & 0  
         \end{tikzcd}
         &
         \\ 
         \hline 
         $\ul w$
         &
         $\ul w^*$
         &
         \\ \hline
         \begin{tikzcd}[]
         & 0  & \\
         \F \ar[dr,"1" swap] & \F \ar[d,"1"'] & \F \ar[dl,"1"]\\
          & \F     
         \end{tikzcd}
         &
         \begin{tikzcd}[swap]
         & 0    & \\
         \F   & \F  & \F  \\
          & \F  \ar[ur,"1"] \ar[u,"1"] \ar[ul,"1"']
         \end{tikzcd}
         &
\\
\hline
\end{tabular} }
\end{center}
\end{table}

The  homotopy Mackey functors of $\Sigma^{n\rho} H_K \ulZ$ were computed in \cite{Slone}*{Section~9}. They are displayed in \cref{fig:HtpySigmarhoK4Z}.
The  homotopy Mackey functors of $\Sigma^{n\rho} H_K \ulF$ were computed in \cite{GY}*{Section~7}. They are displayed in \cref{fig:HtpySigmarhoK4}.

%%%%%%%%%%%%%%%%%%%%%%%%%%%%%%%%%%%%%%%%%%%%%%%%%%%%%
%%%%%%%%%%%%%%   Sigma^{n\rho} H_K Z  %%%%%%%%%%%%%%%
%%%%%%%%%%%%%%%%%%%%%%%%%%%%%%%%%%%%%%%%%%%%%%%%%%%%%

\begin{figure}
\ifDrawTikzPictures
\else
\includegraphics{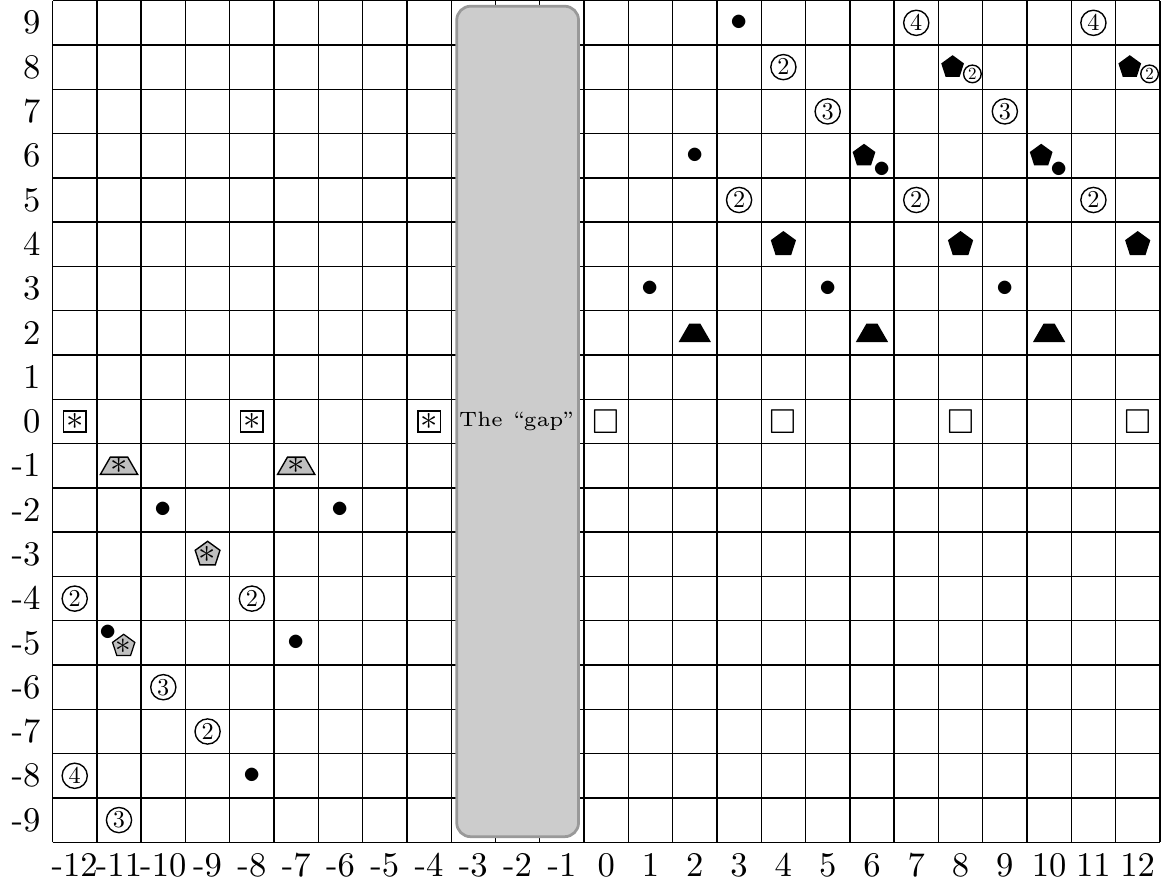}
\fi
\caption{
The homotopy Mackey functors of $\bigvee_n \Sigma^{n\rho} H_{K_4} \ulZ$. The Mackey functor $\mpi_k \Sigma^{n\rho} H_{K_4} \ulZ$ appears in position $(k,4n-k)$.
}
\label{fig:HtpySigmarhoK4Z}
\end{figure}

%%%%%%%%%%%%%%%%%%%%%%%%%%%%%%%%%%%%%%%%%%%%%%%%%%%%%
%%%%%%%%%%%%%%   Sigma^{n\rho} H_K F  %%%%%%%%%%%%%%%
%%%%%%%%%%%%%%%%%%%%%%%%%%%%%%%%%%%%%%%%%%%%%%%%%%%%%
\begin{figure}
\ifDrawTikzPictures
\else
\includegraphics{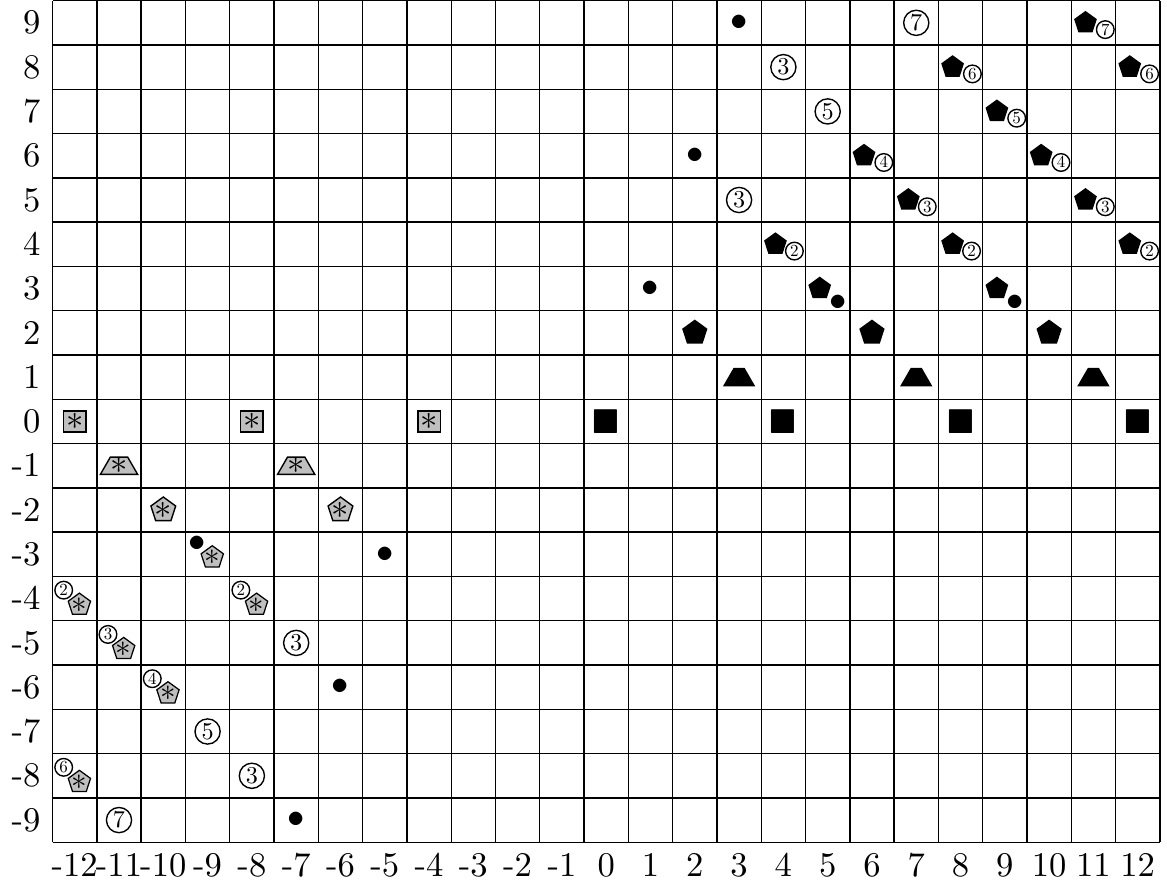}
\fi
\caption{
The homotopy Mackey functors of $\bigvee_n \Sigma^{n\rho} H_{K_4} \ulF$. The Mackey functor $\mpi_k \Sigma^{n\rho} H_{K_4} \ulF$ appears in position $(k,4n-k)$.
}
\label{fig:HtpySigmarhoK4}
\end{figure}

\subsection{Background for $Q_8$}

The regular representation of $Q$ splits as 
\[ \rho_Q \iso \QH \oplus \rho_K,\]
where $\QH$ is the 4-dimensional irreducible $Q_8$-representation given by 
the action of the unit quaternions on the  algebra of quaternions  
and $\rho_K$ is the regular representation of $K$, inflated to $Q$ along the quotient. 

Denoting by $C_4$ any of the subgroups $L$, $D$, or $R$ of $Q_8$, we have that 
\[ 
\downarrow^{Q_8}_{C_4} \rho_K = 2 + 2\sigma \qquad \text{and} \qquad 
\downarrow^{Q_8}_{C_4} \QH = 2\lambda.
\]

%%%%%%%%%%%%%%%%%%%%%%%%%%%%%%%
\section{Inflation functors}
\label{Inflationsection}

%%%%%%%%%%%%%%%%%%%%%%%%%%%%%%%%%%%%%%%%%%%%%%%%%%%%%%%%
\subsection{Inflation and the projection formula}
\label{sec:Projection}

Let $N \unlhd G$ be a normal subgroup and $q\colon G \rtarr G/N$ the quotient map. Recall that there is an induced adjunction
\[
\begin{tikzcd}[every arrow/.append style={shift left}]
\Sp^{G/N} \ar[r,"q^*"] &
\Sp^{G} \ar[l,"(-)^N"]
\end{tikzcd}
\]
where the pullback functor $q^*$, called inflation, is strong symmetric monoidal. 
We will also need a description of the $N$-fixed points of an Eilenberg-Mac~Lane $G$-spectrum. First note that there is a  functor
\begin{equation} 
\label{qstar}
\Mack(G) \xrtarr{q_*} \Mack(G/N)
\end{equation}
given by 
\[
q_*(\ul{M})(\overline{H}) = \ul{M}(H),
\]
where $\overline{H} = H/N \leq G/N$ whenever $N\leq H$. The functor $q_*$ is denoted $\beta^!$ in \cite{TW}*{Lemma~5.4}.
Then the homotopy Mackey functors of the  $N$-fixed points of a $G$-spectrum $X$ are given by
\begin{equation}
\mpi_n(X^N) \iso q_* \mpi_n(X).
\label{HtpyFixedPts}    
\end{equation}
In the case of an Eilenberg-Mac~Lane spectrum this yields an equivalence
\[ (H_G \ul{M})^N \simeq H_{G/N} (q_* \ul{M}).\]
The following result will be quite useful.

\begin{prop} \cite{HKBPO}*{Lemma~2.13}
\cite{BDS}*{Proposition~2.15} 
(Projection formula)
\label{ProjectionFormula}
Let $N \unlhd G$ be a normal subgroup and $q\colon G \rtarr G/N$ be the quotient map. 
Then for $X\in \Sp^{G/N}$ and $Y\in \Sp^{G}$, there is a natural equivalence of $G/N$-spectra
\[ (q^* X \smsh Y)^N \simeq X \smsh Y^N. \]
\end{prop}

We will frequently employ this in the case that $X=S^V$ for some $G/N$-representation $V$ and $Y=H_G \ul{M}$ for some $G$-Mackey functor $\ul{M}$. Then the projection formula reads
\begin{equation}
\label{ProjectionHM}    
(S^{q^* V} \smsh H_G\ul{M})^N \simeq S^V \smsh H_{G/N}(q_*\ul{M}).
\end{equation} 
See also \cite{Z}*{Corollary~5.8}

%%%%%%%%%%%%%%%%%%%%%%%%%%%%%%%%%%%%%%%%%%%%%%%%%%%%%%%%
\subsection{Geometric fixed points}
\label{sec:GeoFixed}

For a normal subgroup $N \unlhd G$, we define the family of subgroups $\mathcal{F}[N]$ of $G$ to consist of those subgroups that do not contain $N$. 
Recall that the $N$-geometric fixed points spectrum of a $G$-spectrum is defined as
\[ \Phi^N(X) = \left( \widetilde{E\mathcal{F}[N]} \smsh X \right)^N.\]
This notation is simultaneously used to denote the resulting $G/N$-spectrum as well as the underlying spectrum. The $N$-geometric fixed points has a right adjoint, given by the geometric inflation functor
\[ \phi^*_N (Z) = \widetilde{E\mathcal{F}[N]} \smsh q^* Z.\]
To sum up, we have an adjunction
\[
\begin{tikzcd}[every arrow/.append style={shift left}]
\Sp^{G} \ar[r,"\Phi^N"] &
\Sp^{G/N}. \ar[l,"\phi_N^*"]
\end{tikzcd}
\]

%%%%%%%%%%%%%%%%%%%%%%%%%%%%%%%%%%%%%%%%%%%%%%%%%%%%%%%%
\subsection{Bottleneck subgroups}
\label{sec:Bottleneck}

The subgroup $Z \unlhd Q$ plays an important role in this article. The primary reason is that it satisfies the following property.

\begin{defn}
We say that $N \unlhd G$ is a {\bf bottleneck} subgroup if it is a nontrivial, proper subgroup such that, for any subgroup $H \leq G$, either $H$ contains $N$ or $N$ contains $H$.
\end{defn}

We now demonstrate that bottleneck subgroups only occur in cyclic p-groups or quaternion groups. The following argument was sketched to us by Mike Geline.

\begin{prop}
\label{ShapeofBottlenecks}
Let $N \unlhd G$ be a  bottleneck subgroup of $G$. Then $N$ is cyclic, and $G$ is either a cyclic $p$-group or a generalized quaternion group.
\end{prop}

\begin{pf}
We will refer to a subgroup $H\leq G$ which neither contains $N$ nor is contained in $N$ as ``adjacent'' to $N$.
The assumption that $N$ is a bottleneck subgroup means precisely that $G$ has no  subgroups that are adjacent to $N$.
To see that $N$ must be cyclic, note 
that if $g$ is not in $N$, then $N \leq \langle g \rangle$, which implies that $N$ is cyclic.

We next observe that $G$ is necessarily a $p$-group. This is because if $N$ is contained in some Sylow $p$-subgroup, then any Sylow $q$-subgroup, for a different prime $q$, would be adjacent. It follows that $N$ contains all of the Sylow subgroups and therefore is all of $G$.

Next, we recall \cite{Brown}*{Theorem~4.3} that for a $p$-group $G$, the group contains a unique subgroup of order $p$ if and only if $G$ is either cyclic or generalized quaternion. So we will argue that $G$ contains a unique subgroup of order $p$. The first step is to note that $G$ cannot contain a subgroup isomorphic to $C_p \times C_p$. This is because such a subgroup would necessarily contain $N$. This would imply that $N \iso C_p$, and then $N$ would have a complement in $C_p \times C_p$, which would be a subgroup adjacent to $N$ in $G$.

Finally, note that the center $Z(G)$ contains a subgroup of order $p$. If $G$ has another subgroup of order $p$, these two would generate a $C_p \times C_p$, contradicting the previous step.
\end{pf}

\begin{rmk}
\label{ShapeGmodN}
It follows from \cref{ShapeofBottlenecks} that if $N \unlhd G$ is a bottleneck subgroup, then 
$G/N$ is either a cyclic $p$-group or a dihedral 2-group.
\end{rmk}

If $N\unlhd G$ is a bottleneck subgroup, then geometric fixed points with respect to $G$ can be computed in terms of geometric fixed points with respect to the quotient group $G/N$.

\begin{prop}
\label{GeoFixedBottleneck}
Let $N\unlhd G$ be a bottleneck subgroup. Then $\Phi^G X \simeq \Phi^{G/N} X^N$ for any $X\in \Sp^G$.
\end{prop}

\begin{pf}
If $N\unlhd G$ is a bottleneck subgroup,
then $q^* \widetilde{E\mathcal{P}_{G/N}} \simeq \widetilde{E \mathcal{P}_G}$. Thus
\[ \Phi^G X = ( \widetilde{E \mathcal{P}_G} \smsh X)^G \simeq (( q^*\widetilde{E \mathcal{P}_{G/N}} \smsh X)^N )^{G/N}.\]
By the Projection Formula (\cref{ProjectionFormula}), this is equivalent to
\[\ (\widetilde{E \mathcal{P}_{G/N}} \smsh X^N)^{G/N} = \Phi^{G/N} X^N. \qedhere
\] 
\end{pf}

\cref{GeoFixedBottleneck}
also follows from the  more general \cite{SKriz}*{Proposition~9}.

%%%%%%%%%%%%%%%%%%%%%%%%%%%%%%%
\subsection{Inflation for $\ulZ$-modules}
\label{sec:Inflation}

Given a surjection \( q\colon G \rtarr G/N \), the inflation functor 
\[
\phi_N^* \colon \Mack(G/N) \rtarr \Mack(G)
\]
does not send $\ulZ$-modules for $G/N$ to $\ulZ$-modules for $G$. We now describe a modified inflation functor that exists at the level of $\ulZ$-modules. This functor  previously appeared in \cite{Z}*{Section~3.2} and \cite{BasuGhosh}*{Section~3.10}.

\begin{defn}\label{defn:PsiN}
Let $\mathcal{B}\ulZ_G \subset \Mod_{\Z[G]}$ denote the full subcategory of permutation $G$-modules. Recall \cite{Z}*{Proposition~2.15} that $\ulZ_G$-modules correspond to additive functors $\mathcal{B}\ulZ_G^{op} \rtarr \mathrm{Ab}$. Then the $\ulZ$-module inflation functor
\[ \phiZ^*_N \colon \Mod_{\ulZ_{G/N}} \rtarr \Mod_{\ulZ_{G}}
\]
is defined to be the left Kan extension along the inflation functor $\mathcal{B}\ulZ_{G/N} \rtarr \mathcal{B}\ulZ_G$.
\end{defn}

The following is an immediate corollary of the definition as a left Kan extension.

\begin{prop}
The functor $\phiZ^*_N$ is left adjoint to the functor
\( q_* \colon \Mod_{\ulZ_G} \rtarr \Mod_{\ulZ_{G/N}}\), defined as in \eqref{qstar}.
\end{prop}

\begin{prop}[\cite{BasuGhosh}*{(3.11)}]
\label{DescriptionphiZ}
For $\ul{M} \in \Mod_{\ulZ_{G/N}}$,
the $\ulZ_G$-module $\phiZ^*_N(\ul{M})$ satisfies
\begin{enumerate}
    \item \(q_* \left( \phiZ^*_N(\ul{M}) \right) \) is $\ul{M}$ and
    \item \(\downarrow^G_N \left( \phiZ^*_N(\ul{M}) \right) \) is the constant Mackey funtor at $\ul{M}(e)$.
\end{enumerate}
\end{prop}

Note that  \cref{DescriptionphiZ} completely describes $\phiZ^*_N(\ul{M})$ if $N$ is a bottleneck subgroup.
The following result states that $\ulZ$-module inflation agrees with ordinary inflation on geometric Mackey functors.

\begin{prop}
\label{phiZgeometric}
Let $\ul{M} \in \Mod_{\ulZ_{G/N}}$, and let $N\unlhd G$ be a bottleneck subgroup.
If $\ul{M}(e)=0$, then $\phiZ_N^* \ul{M} \iso \phi_N^* \ul{M}$.
\end{prop}

\begin{pf}
This follows immediately from \cref{DescriptionphiZ}.
\end{pf}

\begin{rmk}
Note that \cref{phiZgeometric} is not true without the bottleneck hypothesis. For instance, in the case $N=C_3 \unlhd \SI_3$, then $\downarrow^{\SI_3}_{C_2} \left( \phiZ_{C_3}^* \ul{M} \right) \iso \ul{M}$. In particular, it is not true that $\phiZ_{C_3}^* \ul{M}$ is concentrated over $N=C_3$.
\end{rmk}

We now discuss the extension to equivariant spectra.

\begin{prop}
\label{HZinflation}
The $N$-fixed points functor
\[ 
(-)^N \colon \Mod_{H_G \ulZ} \rtarr \Mod_{H_{G/N}\ulZ}
\]
for $H\ulZ$-modules has a left adjoint
\[
\HphiZ_N^* \colon \Mod_{H_{G/N}\ulZ} \rtarr \Mod_{H_G\ulZ}.
\]
If $N \unlhd G$ is a bottleneck subgroup, 
then the spectrum-level functor $\HphiZ_N^*$ extends the functor $\phiZ_N^*$ of \cref{defn:PsiN}, in the sense that
\begin{equation} \label{PsiExtends}
\HphiZ_N^* H_{G/N} \ul{M} \simeq H_G (\phiZ^*_N \ul{M})
\end{equation}
for $\ul{M}$ in $\Mod_{\ulZ_{G/N}}$.
\end{prop}

\begin{pf}
For an $H_{G/N}\ulZ$-module $X$, the inflation $q^*X$ is canonically a module over $q^* H_{G/N}\ulZ$.
We then define the spectrum-level functor $\HphiZ_N^*$ by the formula
\[ 
\HphiZ_N^* X = H\ulZ \smsh_{q^* H\ulZ} (q^*X).
\]
We leave it to the reader to verify that this is indeed left adjoint to the $N$-fixed points functor.

To see that 
\eqref{PsiExtends} holds, we show first that this holds on the indecomposable projective
$\ulZ_{G/N}$-modules.
These are of the form $\uparrow_{K/N}^{G/N} \ulZ$,
and the diagram of commuting adjoint functors
\[
\begin{tikzcd}
    \Mod_{H_{G/N}\ulZ} \ar[d,"\downarrow^{G/N}_{K/N}",xshift={2pt}] \ar[r,"\HphiZ_N^*",yshift={2pt}]  
    & \Mod_{H_G\ulZ} 
    \ar[d,"\downarrow^G_K",xshift={2pt}] \ar[l,"(-)^N",yshift={-2pt}] 
    \\
    \Mod_{H_{K/N}\ulZ} \ar[u,"\uparrow_{K/N}^{G/N}",xshift={-2pt}] \ar[r,"\HphiZ_N^*",yshift={2pt}] 
    & \Mod_{H_K\ulZ} 
    \ar[l,"(-)^N",yshift={-2pt}]
    \ar[u,"\uparrow^G_K",xshift={-2pt}]
\end{tikzcd}
\]
shows that 
\[
\HphiZ_N^* \left(H_{G/N} \uparrow_{K/N}^{G/N} \ulZ \right)  \simeq \  \uparrow_K^G \HphiZ_N^* (H_{K/N} \ulZ) \simeq \ \uparrow_K^G H_K \ulZ 
\simeq H_G \uparrow_K^G \ulZ
\simeq H_G \, \phiZ^*_N \left( \uparrow_{K/N}^{G/N} \ulZ \right).
\]
Since the functor $\phiZ^*_N\colon \Mod_{\ulZ_{G/N}} \rtarr \Mod_{\ulZ_G}$ is exact \cite{Z}*{Lemma~3.14}, it follows that if $\Mod_{\ulZ_{G/N}}$ has finite global projective dimension, then 
\eqref{PsiExtends} will hold for any $\ulZ_{G/N}$-module $\ul{M}$.
By  \cite{BSW}*{Theorem~1.7}, this is the case precisely when $G/N$ is as described in \cref{ShapeGmodN}.
\end{pf}

\begin{eg}
\label{phiZGeometricSpectrumExample}
Let $X \in \Sp^{G/N}$ and $\ul{M} \in \Mack(G/N)$, with $\ul{M}(e)=0$. 
Again assume that $N$ is a bottleneck subgroup.
Then \cref{phiZgeometric} and \cref{HZinflation} give that
\[\begin{split}
    \HphiZ_N^*(X \smsh H_{G/N}\ul{M} ) &\simeq q^*(X) \smsh \HphiZ_N^* ( H_{G/N} \ul{M} ) \simeq q^*(X) \smsh \phi_N^* H_{G/N} \ul{M} \\
    &\simeq \phi_N^* ( X \smsh H_{G/N} \ul{M} ).
\end{split}
\]
We will employ this equivalence when $X$ is a representation sphere.
\end{eg}

\begin{prop}\label{HtpyInflation}
Let $N\unlhd G$ be a bottleneck subgroup. Then for any $G/N$-representation $V$ and $\ulZ_{G/N}$-module $\ul{L}$, we have
\[
\mpi_n \left( \HphiZ_N^* \Si{V} H_{G/N} \ul{L} \right) \iso \phiZ_N^* \mpi_n \left( \Si{V} H_{G/N} \ul{L}\right).
\]
\end{prop}

\begin{pf}
Let us write
$X = \HphiZ_N^* \Si{V} H_{G/N} \ul{L} \simeq \Si{q^* V} H_G \phiZ_N^* \ul{L}$.
Since $N$ is a bottleneck subgroup,  it is enough to describe $\downarrow^G_N \mpi_n X$ and $q_* \mpi_n X$. Now \[\downarrow^G_N \mpi_n X \iso \mpi_n \downarrow^G_N X = \mpi_n \Si{\dim V} H_N \ul{L}(N/N).
\]
This is a constant Mackey functor. On the other hand, by \cref{HtpyFixedPts} and \cref{ProjectionHM}, we have
\[
q_* \mpi_n X \iso \mpi_n (X^N) \iso \mpi_n (\Si{V} H_{G/N} \ul{L}).
\]
By \cref{DescriptionphiZ}, this agrees with $\phiZ_N^* \mpi_n \left( \Si{V} H_{G/N} \ul{L}\right)$.
\end{pf}

More generally, we have an extension of \cref{phiZgeometric} to $H\ulZ$-modules:

\begin{prop}
\label{HphiZgeometric}
Let $X \in \Mod_{H\ulZ_{G/N}}$ and let $N\unlhd G$ be a bottleneck subgroup.
If the underlying spectrum $\downarrow^{G/N}_e \! X$ is contractible, then  $\HphiZ_N^* (X) \simeq \phi_N^* X$.
\end{prop}

\begin{pf}
If the underlying spectrum of $X$ is contractible, then $X \simeq \widetilde{E (G/N)} \smsh X$.
The assumption that $N$ is a bottleneck subgroup implies that ${E (G/N)}=q^*(E(G/N))$ is the universal space for the family of subgroups of $N$, so that 
$\widetilde{E (G/N)} \smsh \widetilde{ E \mathcal{F}[N]} \simeq \widetilde{E (G/N)}$
and it follows that 
\[q^* X \simeq \widetilde{E(G/N)} \smsh q^* X \simeq \widetilde{E(G/N)} \smsh \phi_N^* \left(  X\right)  \simeq \phi_N^* X .\]
Now
\[\begin{split}
\HphiZ_N^* (X) &= H_G \ulZ \smsh_{q^* H_{G/N} \ulZ} q^*(X) \\
&\simeq  H_G \ulZ \smsh_{q^* H_{G/N} \ulZ} (\widetilde{E(G/N)} \smsh q^*(X)).
\end{split}\]
Since $\widetilde{E(G/N)}$ is smash idempotent, this can be rewritten as 
\[
\HphiZ_N^* (X) \simeq 
\widetilde{E(G/N)} \smsh H_G \ulZ \smsh_{\widetilde{E(G/N)} \smsh q^* H_{G/N} \ulZ} \widetilde{E(G/N)} \smsh q^*(X).
\]
It remains only to show that 
\[
\widetilde{E(G/N)} \smsh H_G \ulZ \simeq {\widetilde{E(G/N)} \smsh q^* H_{G/N} \ulZ}.
\]
Both sides restrict trivially to an $N$-equivariant spectrum, so it suffices to show an equivalence on $\Phi^H$, where $H$ properly contains $N$. Without loss of generality, we may suppose $H=G$. Since $\Phi^G( \widetilde{E(G/N)} ) \simeq S^0$, it suffices to show that 
\[ \Phi^G H_G \ulZ \simeq \Phi^G q^* H_{G/N} \ulZ.\]
According to \cref{GeoFixedBottleneck}, the left side is $\Phi^{G/N} H_{G/N}\ulZ$. Similarly, \cref{GeoFixedBottleneck} and the Projection Formula (\cref{ProjectionFormula}) show that the right side is
\[
\begin{split} 
\Phi^G q^* H_{G/N} \ulZ &\simeq \Phi^{G/N} \left( H_{G/N} \ulZ \smsh (S_G^0)^N \right) \\
& \simeq \Phi^{G/N} H_{G/N} \ulZ \smsh \Phi^{G/N} (S^0_G)^N \\
& \simeq \Phi^{G/N} H_{G/N} \ulZ.
\end{split} 
\]
\end{pf}

\begin{thm}
\label{MainSliceInflationTheorem}
Let $n\geq 0$ and let $N \unlhd G$ be a bottleneck subgroup of order $p$, a prime. 
Let $\ul{M} \in \Mod_{\ulZ_{G/N}}$ such that $P^n_n \Si{n} H_{G/N} \ul{M}$ is of the form $\Si{V} H_{G/N} \ul{L}$, for some $G/N$-representation $V$ and $\ul{L} \in \Mod_{\ulZ_{G/N}}$.
Then 
the nontrivial slices of the Eilenberg-Mac~Lane $G$-spectrum $\Si{n} H_G (\phiZ^*_N \ul{M})$, above level $pn$, are
\[ P^{pk}_{pk} \left( \Si{n}H_G (\phiZ^*_N \ul{M}) \right) \simeq \HphiZ_N^* P^k_k \left( \Si{n} H_{G/N} \ul{M} \right)
\simeq \phi_N^* P^k_k \left( \Si{n} H_{G/N} \ul{M} \right)
\]
for $k>n$. Furthermore,
\[
P^{pk}_n \left( \Si{n}H_G ( \phiZ^*_N \ul{M} ) \right) \simeq \HphiZ_N^* P^k_n \left( \Si{n} H_{G/N} \ul{M} \right).
\]
\end{thm}

\begin{pf}
Applying the functor $\HphiZ_N^*$ to the slice tower for $\Si{n}H_{G/N}  \ul{M}$ produces a tower of fibrations whose layers are $\HphiZ_N^* P^k_k \left( \Si{n} H_{G/N} \ul{M} \right)$ for $k \geq n$.
We wish to show that this is a partial slice tower for $\Si{n}H_G (\phiZ^*_N \ul{M})$.
For $k > n$, the $k$-slice $P^k_k \left( \Si{n}H_{G/N}  \ul{M} \right)$ has trivial underlying spectrum. It follows from \cref{HphiZgeometric} that
\[
 \HphiZ_N^* P^k_k \left( \Si{n} H_{G/N} \ul{M} \right)
\simeq 
\phi_N^* P^k_k \left( \Si{n} H_{G/N} \ul{M} \right)
 \]
 for $k > n$. As the geometric inflation of a $k$-slice, this is a $pk$-slice. 
 
It remains to show that 
\[\HphiZ_N^* P^n_n \left( \Si{n} H_{G/N} \ul{M} \right)
\simeq \HphiZ_N^* \Si{V} H_{G/N} \ul{L} \simeq \Si{V} H_G \phiZ_N^* \ul{L}
\]
has no slices above level $pn$.
First, note that the restriction of 
$\Si{V} H_G \phiZ_N^* \ul{L}$
to $N$ is the $N$-spectrum $\Si{n} H_N \ul{L(N)}$, where $\ul{L(N)}$ is being considered as a constant $N$-Mackey functor at the value $\ul{L}(G/N)$. It follows that this $N$-spectrum has no slices above dimension $|N|\cdot n = pn$. 
Therefore, to show that $\Si{V} H_G \phiZ_N^* \ul{L}$ is less than $pn$, it suffices to show that
\[
[ G_+\smsh_H S^{k\rho_H+r} ,  \Si{V} H_G \phiZ_N^* \ul{L}]^G=0
\]
for any $N < H \leq G$ and integers $r\geq 0$ and $k$ such that $k|H| > pn$. Without loss of generality we consider the case $H=G$.
 
 Denote by $U$ a complement of $\rho_{G/N}$ in $\rho_G$, so that 
 \[ 
 \rho_G \iso \rho_{G/N} \oplus U.
 \] 
We then have a cofiber sequence
\[ 
S(kU)_+ \smsh S^{k\rho_{G/N}} \rtarr S^{k\rho_{G/N}} \rtarr S^{k\rho_G}
\]
and a resulting exact sequence
\[\begin{split}
[\Si1 S(kU)_+ \smsh S^{k\rho_{G/N}+r} ,  \Si{V} H_G \phiZ_N^* \ul{L}]^G &\rtarr
[  S^{k\rho_G+r} ,  \Si{V} H_G \phiZ_N^* \ul{L}]^G \\
&\rtarr 
[  S^{k\rho_{G/N}+r} ,  \Si{V} H_G \phiZ_N^* \ul{L}]^G = 0.    
\end{split}
\]
We must show that the left term vanishes. Note that the $G$-action on $S(kU)$ is free, since $N$ is order $p$. Then the desired vanishing follows from the fact that $\Si1 S(kU)_+ \smsh S^{k\rho_{G/N}-V}$ is $G$-connected, since $\dim k\rho_{G/N} > \dim V = n$.
\end{pf}

%%%%%%%%%%%%%%%%%%%%%%%%%%%%%%%
\section{$Q_8$-Mackey functors and Bredon homology}
\label{Qsection}

We display a number of the $Q_8$-Mackey functors that will be relevant in \cref{tab-Q8Mackey}. 
In these Lewis diagrams, we are using the subgroup lattice of $Q_8$ as displayed in \cref{sec:notn}.
We will also often abuse notation and write the name for a $K_4$-Mackey functor, such as $\ulm$ or $\ul{mg}$, to denote the resulting inflated $Q_8$-Mackey functor. We will only write the symbol $\phi_Z^*$ when it is necessary to resolve an ambiguity, for instance between $\phi_Z^* \ulF$ and $\ulF$.

In \cite{HHR2}*{Section~2.1}, the authors introduce ``forms of $\ulZ$'' Mackey functors $\ulZ(i,j)$, where $i \geq j \geq 0$, in the case of $G=C_{p^n}$. From our point of view, $Q_8$ behaves very similarly to $C_8$, and we similarly write $\ulZ(i,j)$ for the Mackey functor that looks like $\ulZ^*$ between the subgroups of order $2^i$ and $2^j$ and looks like $\ulZ$ outside of this range. We will at times follow \cite{HHR2} in denoting by $\ul{B}(i,j)$ the cokernel of $\ulZ(i,j) \into \ulZ$, although we will often instead use the descriptions given in \cref{MackeySESExamples}.

\begin{table}%[h] 
\caption{Some $Q_{8}$-Mackey functors}
\label{tab-Q8Mackey}
\begin{center}
{\renewcommand{\arraystretch}{1.5}
\begin{tabular}{|c|c|c|}
\hline 
$\square = \ulZ$ &
$\dbox = \ulZ^*$ & 
         $\circ=\ul B(3,0) $ \\
         \hline
         \begin{tikzcd}[bend right=2.5ex, swap]
         & \Z \ar[dl,"1", color=blue] \ar[d,"1", color=blue] \ar[dr,"1", color=blue] & \\
         \Z \ar[dr,"1", color=blue] \ar[ur,"2", color=orange]  & \Z \ar[d,"1", color=blue] \ar[u,"2", color=orange]  & \Z \ar[ul,"2", color=orange]  \ar[dl,"1", color=blue] \\
          & \Z \ar[d,"1", color=blue] \ar[ul,"2", color=orange] \ar[u,"2", color=orange] \ar[ur,"2", color=orange] \\
          & \Z \ar[u,"2"]
         \end{tikzcd}
         & 
         \begin{tikzcd}[bend right=2.5ex, swap]
         & \Z \ar[dl,"2", color=orange] \ar[d,"2", color=orange] \ar[dr,"2", color=orange] & \\
         \Z \ar[dr,"2", color=orange] \ar[ur,"1", color=blue]  & \Z \ar[d,"2", color=orange] \ar[u,"1", color=blue]  & \Z \ar[ul,"1", color=blue]  \ar[dl,"2", color=orange] \\
          & \Z \ar[d,"2", color=orange] \ar[ul,"1", color=blue] \ar[u,"1", color=blue] \ar[ur,"1", color=blue] \\
          & \Z \ar[u,"1", color=blue]
         \end{tikzcd}
         &
         \begin{tikzcd}[bend right=2.5ex, swap]
         & \Z/8 \ar[dl,"1"] \ar[d,"1"] \ar[dr,"1"] & \\
         \Z/4 \ar[dr,"1"] \ar[ur,"2"]  & \Z/4 \ar[d,"1"] \ar[u,"2"]  & \Z/4 \ar[ul,"2"]  \ar[dl,"1"] \\
          & \Z/2  \ar[ul,"2"] \ar[u,"2"] \ar[ur,"2"] \\
          & 0 
         \end{tikzcd}
         \\
         \hline
        $\ulZ(3,2) = \phiZ_Z^* \ulZ(2,1)$
        &
        $\ulZ(3,1) = \phiZ_Z^* \ulZ^*$
        &
         $\phiZM = \phi_Z^*(\ul B(2,0))$
         \\ \hline
         \begin{tikzcd}[bend right=2.5ex, swap]
         & \Z \ar[dl,"2", color=orange] \ar[d,"2", color=orange] \ar[dr,"2", color=orange] & \\
         \Z \ar[dr,"1", color=blue] \ar[ur,"1", color=blue]  & \Z \ar[d,"1", color=blue] \ar[u,"1", color=blue]  & \Z \ar[ul,"1", color=blue]  \ar[dl,"1", color=blue] \\
          & \Z \ar[d,"1", color=blue] \ar[ul,"2", color=orange] \ar[u,"2", color=orange] \ar[ur,"2", color=orange] \\
          & \Z \ar[u,"2", color=orange]
         \end{tikzcd}
         &
         \begin{tikzcd}[bend right=2.5ex, swap]
         & \Z \ar[dl,"2", color=orange] \ar[d,"2", color=orange] \ar[dr,"2", color=orange] & \\
         \Z \ar[dr,"2", color=orange] \ar[ur,"1", color=blue]  & \Z \ar[d,"2", color=orange] \ar[u,"1", color=blue]  & \Z \ar[ul,"1", color=blue]  \ar[dl,"2", color=orange] \\
          & \Z \ar[d,"1", color=blue] \ar[ul,"1", color=blue] \ar[u,"1", color=blue] \ar[ur,"1", color=blue] \\
          & \Z \ar[u,"2", color=orange]
         \end{tikzcd}
         &
         \begin{tikzcd}[bend right=2.5ex, swap]
         & \Z/4 \ar[dl,->>] \ar[d,->>] \ar[dr,->>] & \\
         \Z/2  \ar[ur,"2"]  & \Z/2 \ar[u,"2"] & \Z/2 \ar[ul,"2"]  \\
          & 0   \\
          & 0 
         \end{tikzcd}
        \\
        \hline 
        $\phiZF = \phi_Z^* \ulF$
        & 
         $\phiZFstar = \phi_Z^* \ulF^*$
        & 
         $\mysterysymbol = \mystery$
        \\ \hline
          \begin{tikzcd}[swap]
         & \F \ar[dl,"1"] \ar[d,"1"] \ar[dr,"1"'] & \\
         \F \ar[dr,"1"]  & \F \ar[d,"1"]   & \F  \ar[dl,"1"'] \\
          & \F   \\
          & 0 
         \end{tikzcd}
          &
        \begin{tikzcd}[ swap]
         & \F  & \\
         \F  \ar[ur,"1" swap]  & \F \ar[u,"1"]  & \F \ar[ul,"1"]  \\
          & \F  \ar[ul,"1" swap] \ar[u,"1"] \ar[ur,"1"] \\
          & 0 
         \end{tikzcd}
         &
        \begin{tikzcd}[bend right=2.5ex, swap]
         & \F^2 \ar[dl, "2p_1"] \ar[d, "2\nabla"] \ar[dr, "2p_2" pos=0.75] & \\
         \Z^\sigma\!/4 \ar[dr,->>] \ar[ur, "i_2 q"  pos=0.25 ]  & \Z^\sigma\!/4 \ar[d,->>] \ar[u, "\Delta q"]  & \Z^\sigma\!/4 \ar[ul, "i_1 q"]  \ar[dl, ->>] \\
          & \Z/2  \ar[ul,"2"] \ar[u,"2"] \ar[ur,"2"] \\
          & 0 
         \end{tikzcd}
         \\
         \hline
\end{tabular} }
\end{center}
\end{table}

These Mackey functors fit together in exact sequences as follows:

\begin{prop}
\label{MackeySESExamples}
There are exact sequences of Mackey functors
\begin{enumerate}
    \item $\ulZ(3,2) \into \ulZ \onto \ulg$
    \item $\ulZ(3,1) \into \ulZ \onto \phi_Z^* \ul B(2,0)$
    \item $\ulZ(3,1) \into \ulZ(3,2) \onto \ulm^*$
    \item $\ulZ(2,1) \into \ulZ \onto \ulm$
    \item $\ulZ(1,0) \into \ulZ \onto \phi_Z^* \ulF$
    \item $\ulZ^* \into \ulZ \onto  \ul B(3,0)$
    \item $\ul{mg} \into \mystery \onto \ul{w}$.
\end{enumerate}
\end{prop}

%%%%%%%%%%%%%%%%%%%%%%%%%%%%%%%%%%%%%%%%%%%%%%%%%%%%%%%%
\subsection{$RO(Q_8)$-graded Mackey functor $\ulZ$-homology of a point}

We will now compute the homology of $S^{k\rho_Q}$, with coefficients in $\ulZ$, as a Mackey functor.
The starting point is that the regular representation of $Q$ splits as 
\[ \rho_Q \iso \QH \oplus \rho_K,\]
where $\QH$ is the 4-dimensional irreducible $Q$-representation given by 
the action of the unit quaternions on the  algebra of quaternions  
and $\rho_K$ is the regular representation of $K$, inflated to $Q$ along the quotient. We begin by computing the homology of $S^{k\QH}$. See also \cite{Lu}*{Section~2} for an alternative viewpoint.

First, \cref{ProjectionFormula} and \cite{Slone}*{Proposition~9.1} combine to yield the following.

\begin{prop}\label{HomologySkrhoK}
For $k\geq 0$, the nontrivial homotopy Mackey functors of $\Sigma^{k \rho_K} H_Q \ulZ$ are
\[ 
\mpi_n \left( \Sigma^{k \rho_K} H_Q \ulZ \right) \iso
\begin{cases}
\ulZ & n = 4k \\
 \ul{mg} & n = 4k-2 \\
\ulg^{\frac12(4k-n-1)} & n \in [2k,4k-3], n\ \text{odd} \\
\ulg^{\frac12(4k-n-4)} \oplus \phi^*_{LDR}\ulF & n \in [2k,4k-3], n\ \text{even} \\
\ulg^{n-k+1} & n \in [k,2k-1].
\end{cases}
\]
\end{prop}

Next, we employ the cofiber sequence
\begin{equation} \label{SHcofib}
    S(\QH)_+ \rtarr S^0 \rtarr S^\QH
\end{equation} 
to obtain the homology of $S^{\rho_Q}$ from that of $S^{\rho_K}$.

\begin{figure}[h]
\ifDrawTikzPictures
\else
\includegraphics{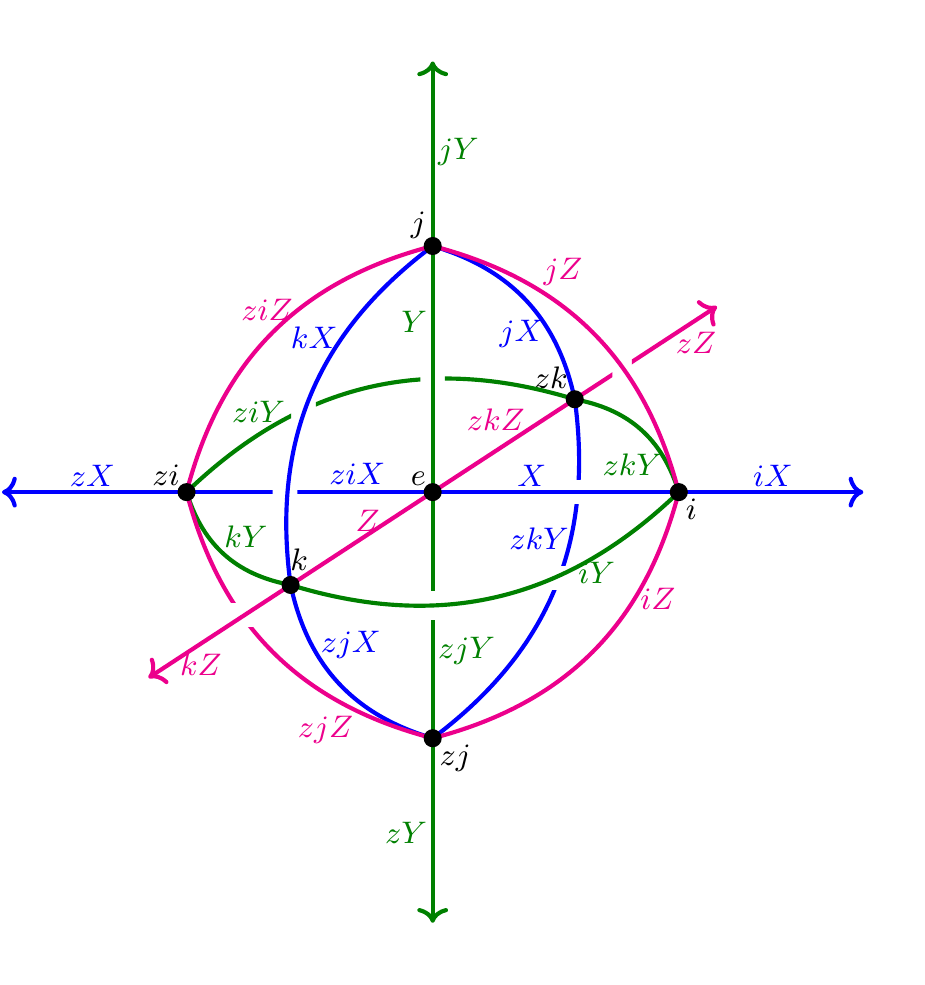}
\fi
\caption{The $1$-skeleton of $S(\QH)$.}
\label{QHFigure}
\end{figure}

\begin{prop}
\label{HomologyS(H)}
The nontrivial homotopy Mackey functors of $S(\QH) \smsh H_Q \ulZ$ are
\[
\mpi_n \left( S(\QH)_+ \smsh H_Q \ulZ \right) \iso
\begin{cases}
\ulZ & n = 3 \\
 \mystery & n = 1 \\
\ulZ^* & n = 0.
\end{cases}
\]
\end{prop}

\begin{pf}
Since the action of $Q$ on $S(\QH)$ is free, we can write down an equivariant cell structure using only free cells. 
Viewing $S(\QH)$ as the one-point compactification of $\R^3$, there is a straight-forward cell structure in which the subgroups $L$, $D$, and $R$ act freely on the $x$, $y$, and $z$-axes, respectively. We display the 1-skeleton in \cref{QHFigure}, and the cell structure is described by the following complex of $\Z[Q]$-modules:
\[ 
\Z[Q]^2 \xrtarr{\scalebox{0.6}{$
\begin{pmatrix}
e & j \\ 
-e & -i \\
e & k \\
-e & -e
\end{pmatrix}
$}}
\Z[Q]^4 \xrtarr{\scalebox{0.6}{$
\begin{pmatrix}
k & e & e & k \\ 
-e & -e & i & i \\
e & -j & -e & j
\end{pmatrix}
$}}
\Z[Q]^3 \xrtarr{(i-e\ j-e\ k-e)} \Z[Q].
\]
This yields an associated complex of induced Mackey functors
\[
\ul{\Z[Q]}^2 \rtarr
\ul{\Z[Q]}^4 \rtarr
\ul{\Z[Q]}^3 \rtarr
\ul{\Z[Q]}
\]
leading to the claimed homology Mackey functors.
\end{pf}

\begin{rmk}
A smaller chain complex for computing the homology of $S(\QH)$ is given by
\[ 
\Z[Q] \xrtarr{\scalebox{0.6}{$\begin{pmatrix}i-e \\ e-k\end{pmatrix}$}}
\Z[Q]^2 \xrtarr{\scalebox{0.6}{$\begin{pmatrix}e+i  & e+k \\ -e-j & -e+i\end{pmatrix}$}}
\Z[Q]^2 \xrtarr{(i-e\ j-e)} \Z[Q].
\]
We gave a less efficient chain complex in the proof of \cref{HomologyS(H)} for geometric reasons.
\end{rmk}

Using \cref{SHcofib}, this immediately yields the following.

\begin{cor}\label{HomologyS^H} 
The nontrivial homotopy Mackey functors of $\Sigma^\QH H_Q \ulZ$ are
\[ 
\mpi_n \left( \Sigma^{\QH} H_Q \ulZ \right) \iso
\begin{cases}
\ulZ & n = 4 \\
 \mystery & n = 2 \\
\ul B(3,0) & n = 0.
\end{cases}
\]
\end{cor}

We will use this to compute the homology of $S^{\rho_Q}$, using the following periodicity result.

\begin{prop}[\cite{Waner}*{Proposition~4.1}]
\label{periodicity}
For any orientable representation $V$ of dimension $d$ and free $Q$-space $X$, the orientation $u_V \in \rH_d(S^V;\ulZ)$ induces an equivalence
\[
\Sigma^d X_+ \smsh H_Q \ulZ \simeq
\Sigma^V X_+ \smsh H_Q \ulZ
\]
\end{prop}

We now compute the homology of $S^{\rho_Q}$.

\begin{prop}
The nontrivial homotopy Mackey functors of $\Sigma^{\rho_Q} H_Q \ulZ$ are
\[ 
\mpi_n \left( \Sigma^{\rho_Q} H_Q \ulZ \right) \iso
\begin{cases}
\ulZ & n = 8 \\
\mystery & n = 6 \\
\ul B(3,0) & n = 4 \\
 \ul{mg} & n = 2 \\
\ulg & n=1.
\end{cases}
\]
\end{prop}

\begin{pf}
The representation $\rho_K$ is orientable. For example, using the basis $\{1,i,j,k\}$ for $\rho_K = \R[K]$, the matrix $\rho_K(i)$ is given by 
\[
\rho_K(i) = 
\begin{pmatrix} 
0 & -1 & 0 & 0 \\ 
1 & 0 & 0 & 0 \\
0 & 0 & 0 & -1 \\
0 & 0 & 1 & 0 
\end{pmatrix},\]
which has determinant equal to 1.
By \cref{periodicity}, we have
\[ 
\mpi_n \left( S(\QH)_+ \smsh \Sigma^{\rho_K} H_Q \ulZ \right) \iso
\begin{cases}
\ulZ & n = 7 \\
\mystery & n = 5 \\
\ulZ^* & n = 4.
\end{cases}
\]
The result then follows from the cofiber sequence
\[ S(\QH)_+ \smsh \Sigma^{\rho_K} H_Q \ulZ \rtarr \Sigma^{\rho_K} H_Q \ulZ \rtarr \Sigma^{\rho_Q} H_Q \ulZ.\]
\end{pf}

\cref{HomologyS^H} generalizes as follows.

\begin{prop}\label{HomologyS^nH} 
The nontrivial homotopy Mackey functors of $\Sigma^{k\QH} H_Q \ulZ$, for $k > 0 $ are
\[ 
\mpi_n \left( \Sigma^{k\QH} H_Q \ulZ \right) \iso
\begin{cases}
\ulZ & n = 4k \\
\mystery & 0 < n < 4k,  n\equiv 2\pmod4 \\
\ul B(3,0) & 0 \leq n < 4k, n\equiv 0\pmod4.
\end{cases}
\]
\end{prop}

\begin{pf}
This follows by induction, using the cofiber sequence
\[ S(\QH)_+ \smsh S^{(k-1)\QH} \rtarr S^{(k-1) \QH} \rtarr S^{k\QH}\]
and \cref{periodicity}. The latter applies since $\QH$, and therefore also $(k-1)\QH$, is orientable. 
\end{pf}

Combining this with the cofiber sequence
\[ S(k\QH)_+ \smsh \Sigma^{k\rho_K} H_Q \ulZ \rtarr \Sigma^{k\rho_K} H_Q \ulZ \rtarr \Sigma^{k\rho_Q} H_Q \ulZ\]
and 
\cref{periodicity} 
gives the following result.

\begin{prop}\label{HomologySkrhoQ}
The nontrivial homotopy Mackey functors of $\Sigma^{k\rho_Q} H_Q \ulZ$, for $k>0$, are
\[ 
\mpi_n \left( \Sigma^{k\rho_Q} H_Q \ulZ \right) \iso
\begin{cases}
\ulZ & n = 8k \\
\mystery & 4k <n < 8k, n\equiv 2\pmod4 \\
\ul B(3,0) &  4k \leq n < 8k, n\equiv 0\pmod4  \\
\phi_Z^* \mpi_n \left( \Sigma^{k\rho_K} H_K \ulZ \right) & n < 4k,\\
\end{cases}
\]
where the latter Mackey functors are listed in \cref{HomologySkrhoK}.
\end{prop}

The homotopy Mackey functors of $\Sigma^{k\rho_Q} H_Q \ulZ$ are displayed in
\cref{fig:HtpySigmarhoQZ}.
When $k$ is negative, the computation follows the same strategy. The initial input, which can again be computed using the chain complex given in \cref{HomologyS(H)}, 
is that
\begin{equation}
\label{CohomS(H)}
\ul{\rH}^n ( S(\QH); \ulZ) \iso
\mpi_{-n} \left( F\big(S(\QH)_+ , H_Q \ulZ \big) \right) \iso
\begin{cases}
\ulZ^* & n = 3 \\
\mystery & n = 2 \\
\ulZ & n = 0.
\end{cases}
\end{equation}

Using this and \cite{Slone}*{Proposition~9.2} leads to the following answer.

\begin{prop}\label{HomologySnegkrhoQ}
The nontrivial homotopy Mackey functors of $\Sigma^{-k\rho_Q} H_Q \ulZ$, for $k>0$, are
\[ 
\mpi_{-n} \left( \Sigma^{-k\rho_Q} H_Q \ulZ \right) \iso
\begin{cases}
\ulZ^* & n = 8k \\
\mystery & n\in [4k,8k], n\equiv 3\pmod4 \\
\ul B(3,0) &  n \in [4k+5,8k], n\equiv 1\pmod4  \\
\phi_Z^* \ul B(2,0) & n=4k+1 \\
\ul{mg}^* & n = 4k-1 \\
\ulg^{\frac{4k-n}2} & n \in [2k+4,4k-2], n\equiv0\pmod2 \\
\ulg^{\frac{4k-n-3}2} \oplus \phi_{LDR}^* \ulF^* & n \in [2k+3,4k-2], n\equiv1\pmod2 \\
\ulg^{n-k-3} & n\in [k+4,2k+2].
\end{cases}
\]
\end{prop}

\begin{rmk}
The ``Gap Theorem''
\cite{Kervaire}*{Proposition~3.20} 
predicts that the groups $\pi^Q_n \Sigma^{-k\rho} H\ulZ$ vanish for $k\geq 0$ and $n\in [-3,-1]$, as indicated in \cref{fig:HtpySigmarhoQZ}. Actually, for $k\geq 2$ the argument there proves more. It tells us that 
for $k\geq 2$, the cohomology groups $\rH^n_Q(S^{k\rho}; \ul{M})$ vanish for positive $n\leq k+1$. This is equivalent to saying that $\pi^Q_{-n} \Sigma^{-k\rho} H \ul{M}$ vanishes, with the same conditions on $k$ and $n$.
\end{rmk}

\subsection{Additional homology calculations}

We will also need the following auxiliary calculations in \cref{sec:Q8slice}.

\begin{prop}
\label{AuxZ}
The nontrivial homotopy Mackey functors of
$\Si{\rho_K-\QH}H_Q \ulZ$ are
\[
\mpi_n\left(
\Si{\rho_K-\QH}H_Q \ulZ
\right) \iso
\begin{cases}
\phi_Z^* \ulF & n=1 \\
\ulZ^* & n=0.
\end{cases}
\]
\end{prop}

\begin{pf}
The fiber sequence
\[
\Si{\rho_K-\QH} H_Q\ulZ \rtarr \Si{\rho_K}H_Q\ulZ \rtarr F(S(\QH)_+,\Si{\rho_K}H_Q\ulZ) \simeq \Si4 F(S(\QH)_+,H_Q\ulZ)
\]
 yields an isomorphism $\mpi_0 \left(
\Si{\rho_K-\QH}H_Q \ulZ
\right) \iso \ulZ^*$ and shows that the homotopy vanishes for $n$ outside of $[0,2]$. Given that the restriction to any $C_4$, which is the $C_4$-spectrum $\Si{2+2\sigma-2\lambda} H_{C_4} \ulZ$,
has a trivial $\mpi_2$ \cite{Z}*{Theorem~6.10}, the long exact sequence further shows that $\mpi_2$ vanishes as well, and it implies that we have an extension
\[ \ulw \into \mpi_1 \left(
\Si{\rho_K-\QH}H_Q \ulZ
\right) \onto \ulg. \]
It remains to show this is not the split extension.
The fiber sequence
\[
\uparrow_D^Q \Si{1+2\sigma-2\lambda} H_{C_4}\ulZ \rtarr \Si{1+p_1^* \sigma+p_2^* \sigma-\QH} H_Q \ulZ \rtarr \Si{\rho_K-\QH}H_Q \ulZ
\]
shows that $\mpi_1 \left(
\Si{\rho_K-\QH}H_Q \ulZ
\right)$ injects into
\[
\mpi_0 \left( \uparrow_D^Q \Si{1+2\sigma-2\lambda} H_{C_4}\ulZ \right) \iso \uparrow_D^Q \phi_{C_2}^* \ulF.
\]
It follows that $\mpi_1 \left(
\Si{\rho_K-\QH}H_Q \ulZ
\right) \iso \phi_Z^* \ulF$
\end{pf}

\begin{prop}
\label{AuxZ32}
The nontrivial homotopy Mackey functors of
$\Si{\rho_K-\QH}H_Q \ulZ(3,2)$ are
\[
\mpi_n\left(
\Si{\rho_K-\QH}H_Q \ulZ(3,2)
\right) \iso
\begin{cases}
\ulw & n=1 \\
\ulZ^* & n=0.
\end{cases}
\]
\end{prop}

\begin{pf}
The short exact sequence
\[ \ulZ(3,2) \into \ulZ \onto \ulg \]
gives rise to a cofiber sequence
\[ \Si{\rho_K-\QH} H_Q \ulZ(3,2) \rtarr \Si{\rho_K-\QH} H_Q \ulZ \rtarr \Si{\rho_K-\QH} H_Q \ulg  \simeq \Si1 H_Q \ulg. \]
Using a naturality square, the second map factors as
\[ \Si{\rho_K-\QH} H_Q \ulZ \rtarr \Si{\rho_K} H_Q \ulZ \rtarr \Si1 H_Q \ulg, \]
where the first map is an epimorphism on $\mpi_1$ by the proof of \cref{AuxZ} and the second is an isomorphism on $\mpi_1$. The conclusion follows.
\end{pf}

\begin{prop}
\label{AuxZ20}
The nontrivial homotopy Mackey functors of
$\Si{\QH-\rho_K}H_Q \ulZ(2,0)$ are
\[
\mpi_n\left(
\Si{\QH-\rho_K}H_Q \ulZ(2,0)
\right) \iso
\begin{cases}
\ulZ & n=0 \\
\ulw^* & n=-2.
\end{cases}
\]
\end{prop}

\begin{pf}
This follows from \cref{AuxZ32} by duality. In more detail, \cref{AuxZ32} gives a fiber sequence
\[
\Si1 H_Q\ulw \rtarr \Si{\rho_K-\QH} H_Q \ulZ(3,2) \rtarr H_Q\ulZ^*.
\]
Applying Anderson duality (see \cite{Slone}*{Section~2.2}) gives a fiber sequence
\[
I(\Si1 H_Q \ulw) \ltarr I\left(\Si{\rho_K-\QH} H_Q \ulZ(3,2)\right) \ltarr I(H_Q\ulZ^*),
\]
or in other words
\[
\Si{-1} I(H_Q\ulw) \ltarr \Si{\QH-\rho_K} H_Q \ulZ(2,0) \ltarr  H_Q\ulZ.
\]
But as the Mackey functor $\ulw$ is torsion, the Anderson dual is the desuspension of the Brown-Comenetz dual. In other words, $I(H_Q\ulw) \simeq \Si{-1} I_{\Q/\Z} H_Q \ulw \simeq \Si{-1} H_Q \ulw^*$.
\end{pf}

%%%%%%%%%%%%%%%%%%%%%%%%%%%%%%%%%%%%%%%%%%%%%%%%%%%%%
%%%%%%%%%%%%%%   Sigma^{n\rho} H_Q Z  %%%%%%%%%%%%%%%
%%%%%%%%%%%%%%%%%%%%%%%%%%%%%%%%%%%%%%%%%%%%%%%%%%%%%
\begin{figure}
\ifDrawTikzPictures
\begin{tikzpicture}[scale=0.35,mytrap/.style={
  trapezium, fill,inner xsep=2pt,scale=0.8},mypent/.style={fill = black, regular polygon, regular polygon sides=5, 
 minimum width=0pt, 
 inner sep = 0.3ex,scale=1.7}]
\draw [step = 1,shift={(0.5,0.5)}] (-17,-11) grid (16,14);
\foreach \x in {-16,...,16}
 \node[anchor=north] at (\x+0,-10.5) {\tiny \x};
\foreach \y in {-10,...,14}
 \node[anchor=east] at (-16.5,\y) {\tiny \y};

\node at (16,8) {$\circ$};
\node at (14,10) {$\mysterysymbol$};
\node at (12,12) {$\circ$};
\node[mytrap] at (10,14) {};
\node at (16,0) {$\square$};
\node at (14,2) {$\mysterysymbol$};
\node at (12,4) {$\circ$};
\node at (10,6) {$\mysterysymbol$};
\node at (8,8) {$\circ$};
\node[mytrap] at (6,10) {};
\node at (5,11) {$\bullet$};
\node[mypent] at (4,12) {};
\node[encirc] at (3,13) {2};
\node at (2,14) {$\bullet$};
\node at (8,0) {$\square$};
\node at (6,2) {$\mysterysymbol$};
\node at (4,4) {$\circ$};
\node[mytrap,scale=0.9] at (2,6) {};
\node at (1,7) {$\bullet$};
\node at (0,0) {$\square$};
 \node at (-8,0) {$\dbox$};
 \node at (-7,-1) {$\mysterysymbol$};
 \node at (-5,-3) {$\phiZM$};
 \node at (-16,0) {$\dbox$};
 \node at (-15,-1) {$\mysterysymbol$};
 \node at (-13,-3) {$\circ$};
\node at (-11,-5) {$\mysterysymbol$};
 \node at (-9,-7) {$\phiZM$};
 \node at (-7,-9) {$\atrap$};
\node at (-6,-10) {$\bullet$};
\node at (-15,-9) {$\mysterysymbol$};
\draw[color=gray, fill=gray!40, thick, rounded corners] (-0.625, -10.375) -- (-0.625, 14.375) -- (-4.375,14.375) -- (-4.375,-10.375) -- cycle;
\node at (-2.5,2) {\tiny The ``gap''};
\end{tikzpicture}
\else
\includegraphics{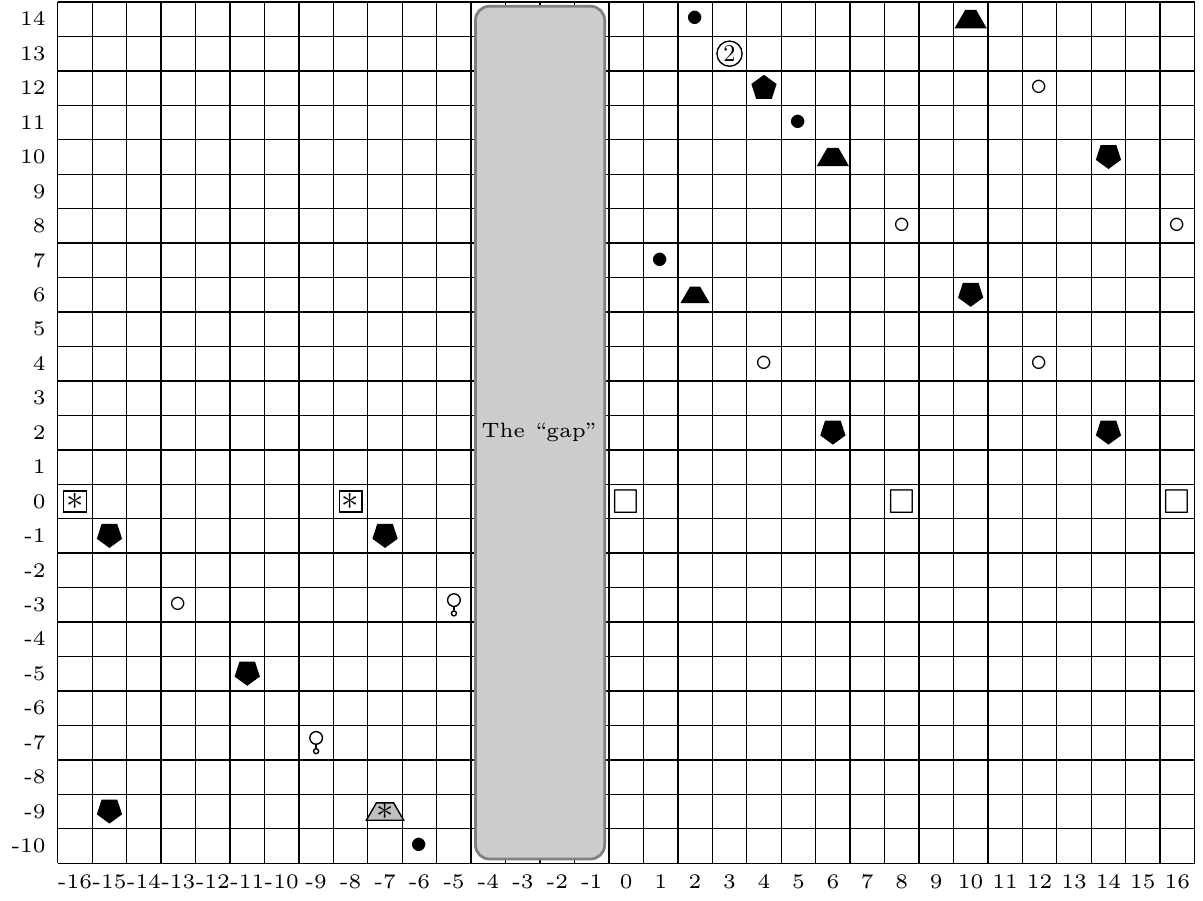}
\fi
\caption{
The homotopy Mackey functors of $\bigvee_n \Sigma^{n\rho} H_{Q} \ulZ$. The Mackey functor $\mpi_k \Sigma^{n\rho} H_{Q} \ulZ$ appears in position $(k,8n-k)$.
}
\label{fig:HtpySigmarhoQZ}
\end{figure}

%%%%%%%%%%%%%%%%%%%%%%%%%%%%%%%
\section{Review of the $C_4$-slices of $\Sigma^n H\Z$}
\label{sec:C4slice}

In this section, we review the slices of $\Si{n}H_{C_4}\ulZ$ from \cite{Yarn}. Note that the slices as listed in \cite{Yarn} are written using the classical slice filtration, whereas we use the regular slice filtration. The only difference is a suspension by one.
The Mackey functors that appear here were introduced in \cref{tab-C4Mackey}.

According to \cite[Section~4.2]{Yarn}, the $C_4$-spectrum $\Sigma^n H_{C_4}\ulZ$ is an $n$-slice for $0\leq n \leq 4$. For \(n\geq 5\), $\Sigma^n H_{C_4}\ulZ$ has a nontrivial slice tower. 
Yarnall's method for determining these slice towers is to splice together suspensions of the cofiber sequences
\[
\Si{-1} H_{C_4} \ulg \rtarr
\Si{2}H_{C_4} \ulZ \rtarr \Si{2\sigma} H_{C_4} \ulZ,
\]
\[
\Si{-1} H_{C_4} \phi_{C_2}^* \ulF^* \rtarr 
\Si{2} H_{C_4} \ulZ \rtarr \Si{\lambda} H_{C_4} \ulZ(2,1),
\]
and
\[
\Si{-1} H_{C_4} \ul{B}(2,0) \rtarr 
\Si{2} H_{C_4} \ulZ \rtarr \Si{\lambda} H_{C_4} \ulZ
\]
in combination with the equivalences
\[
\Si{2} H_{C_4} \ulZ \simeq \Si{2\sigma} H_{C_4} \ulZ(2,1)\]
and
\[ \Si{-1} H_{C_4} \phi_{C_2}^* \ulF^* \simeq \Si{-\sigma} H_{C_4} \phi_{C_2}^* \ulf \simeq \Si{1-2\sigma} H_{C_4} \phi_{C_2}^* \ulF.
\] 

We first review these slices for odd \(n\).

\begin{prop}{\cite[Theorem 4.2.6]{Yarn}}
Let $n\geq 5$ be odd. The bottom slice of $\SI^n H_{C_4} \ulZ$ is
\begin{align*}
    P^{n}_{n}( \SI^n H_{C_4}\ulZ) &\simeq \left\lbrace \begin{array}{ll}
        \SI^{\frac{n-5}{4}\rho + 4+\sigma} H_{C_4}\ulZ & n\equiv 1\pmod 8 \\
        \SI^{\frac{n-3}{4}\rho + 3} H_{C_4}\ulZ & n\equiv 3\pmod 8 \\
        \SI^{\frac{n-5}{4}\rho + 3+2\sigma} H_{C_4}\ulZ & n\equiv 5\pmod 8 \\
        \SI^{\frac{n-3}{4}\rho + 2+\sigma} H_{C_4}\ulZ & n\equiv 7\pmod 8.
    \end{array}\right.
\end{align*}
\end{prop}

\begin{prop}{\cite[Lemma 4.2.5]{Yarn}}
\label{4kSlicesC4nodd}
Let \(n\geq 5\) be odd. The nontrivial  \(4k\)-slices of \(\SI^{n} H_{C_4}\ulZ\) are
\begin{align*}
    P^{4k}_{4k}( \SI^n H_{C_4}\ulZ) &\simeq \left\lbrace \begin{array}{ll}
        \SI^{k\rho} H_{C_4}\ul B(2,0) & 4k\in[n+1, 2(n-3)], \ k \text{ even} \\
        \SI^{k\rho} H_{C_4} \phi^*\ul f & 4k\in[n+1, 2(n-3)], \ k \text{ odd} \\
        \SI^{k\rho} H_{C_4}\ulg & 4k\in[2(n-1), 4(n-3)], \ k \text{ even}.
    \end{array}\right.
\end{align*}
\end{prop}

The $4k$-slices can also be read off of \cite{HHR}*{Figure~3}.
When $n$ is odd, these are the only nontrivial slices of \(\SI^{n} H_{C_4}\ulZ\).

We now recall the slices of \(\SI^{n} H_{C_4}\ulZ\) for even \(n\).

\begin{prop}{\cite[Theorem 4.2.9]{Yarn}}
\label{nsliceSInZK}
Let $n\geq 6$ be even. The bottom slice of $\SI^n H_{C_4} \ulZ$ is
\begin{align*}
    P^{n}_{n}( \SI^n H_{C_4}\ulZ) &\simeq \left\lbrace \begin{array}{ll}
        \SI^{\frac{n-4}{4}\rho + 3+\sigma} H_{C_4}\ulZ & n\equiv 0\pmod 8 \\
        \SI^{\frac{n-6}{4}\rho + 3+3\sigma} H_{C_4}\ulZ & n\equiv 2\pmod 8 \\
        \SI^{\frac{n-4}{4}\rho + 4} H_{C_4}\ulZ & n\equiv 4\pmod 8 \\
        \SI^{\frac{n-6}{4}\rho + 4+2\sigma} H_{C_4}\ulZ & n\equiv 6\pmod 8.
    \end{array}\right.
\end{align*}
\end{prop}

\begin{prop}{\cite[Lemma 4.2.7]{Yarn}}
\label{4kSlicesC4neven}
Let \(n\geq 6\) be even. The nontrivial \(4k\)-slices of \(\SI^{n} H_{C_4}\ulZ\) are
\[
 P^{4k}_{4k}( \SI^n H_{C_4}\ulZ) \simeq 
 \SI^{k} H_{C_4}\ulg, \quad  k \ \text{odd} %\\
\]
for $4k$ in the range $[n+2,4n-12]$.
\end{prop}

Again, the $4k$-slices can also be read off of \cite{HHR}*{Figure~3}.

\begin{prop}{\cite[Theorem 4.2.9]{Yarn}}
\label{4kPlus2SlicesC4neven}
Let \(n\geq 6\) be even. 
The \((4k+2)\)-slices of \(\SI^{n} H_{C_4}\ulZ\) are
\begin{align*}
    P^{8k+2}_{8k+2}( \SI^n H_{C_4}\ulZ) &\simeq \SI^{1+2k\rho} H\phi^*\ulF \\
    P^{8k+6}_{8k+6}( \SI^n H_{C_4}\ulZ) &\simeq \SI^{3+2k\rho} H\phi^*\ulF.
\end{align*}
for $8k+2$ or $8k+6$ in the range $[n+2,2n-6]$
\end{prop}

We may also view these slices through the perspective of the \(\ulZ\)-module inflation functor. By \cref{MainSliceInflationTheorem},
\begin{align*}
    \HphiZ_{C_2}^*: \text{Mod}_{H_{C_2}\ulZ} \rtarr \text{Mod}_{H_{C_4}\ulZ}
\end{align*}
will provide all slices of \(\Si{n} H_{C_4}\) above level \(2n\). Let \(r\equiv n\pmod 4\) with \(3\leq r\leq 6\). It follows from \cite[Proposition 3.5]{Slone} that the slices of \(\Si{n} H_{C_4}\ulZ\) in level at least \(2n+2r-4\)  are
\begin{align*}
    P^{4k}_{4k} \left( \Si{n} H_{C_4} \ulZ \right) \simeq    \HphiZ_{C_2}^* \Si{k} H_{C_2}\ulg \simeq \Si{k} H_{C_4}\ulg
\end{align*}
for \(4k\in [2n+2r-4,4(n-3)]\). The rest of the slices then follow from determining the slices of
\begin{align*}
    \HphiZ_{C_2}^* \Si{\frac{n-r}{2}\rho_{C_2} + r} H_{C_2}\ulZ \simeq \Si{\frac{n+r}{2} + \frac{n-r}{2}\sigma} H_{C_4}\ulZ.
\end{align*}
The slice tower for this $C_4$-spectrum can be found by splicing together the cofiber sequences listed at the start of this section.

%%%%%%%%%%%%%%%%%%%%%%%%%%%%%%%
\section{$Q_8$-slices}
\label{sec:Q8slice}

The slices of $\Si{n}H_K \ulZ$ were determined by the second author in \cite{Slone}*{Section 8}. As 
stated in \cref{MainSliceInflationTheorem},
it follows that the $\ulZ$-module inflation functor 
\[
\HphiZ_Z^* \colon \Mod_{H_K\ulZ} \rtarr \Mod_{H_Q \ulZ}
\]
of \cref{HZinflation} will produce all slices of $\Si{n}H_Q\ulZ$ in degree larger than $2n$, as the inflation of the slices of $\Si{n}H_K\ulZ$ above degree $n$. 

The remaining slices of $\Si{n}H_Q\ulZ$ will be given as the slices of $\HphiZ_Z^* \left(P^n_n (\Si{n}H_K \ulZ)\right)$.
By \cite{Slone}*{Proposition~8.5}, these are of the form 
\[\HphiZ_Z^* \Big( \Si{r + j\rho_K } H_K \ulZ\Big) \simeq \Si{r + j\rho_K} H_Q \ulZ,\]
where $r\in \{3,4,5\}$, if $n \not\equiv 2 \pmod4$. In the case $n\equiv2\pmod4$, the same result states that this is 
\[\HphiZ_Z^* \Big( \Si{2 + j\rho_K}H_K\ulZ(1,0) \Big) \simeq \Si{2 + j\rho_K} H_Q \ulZ(2,1).\] 
But the cofiber sequence (\cref{MackeySESExamples})
\begin{equation}
\Si{1+j\rho_K} H_Q \ulm \rtarr \Si{2+j\rho_K} H_Q \ulZ(2,1) \rtarr \Si{2+j\rho_K} H_Q \ulZ
\end{equation}
reduces the computation of slices of $ \Si{2 + j\rho_K} H_Q \ulZ(2,1) $ 
to the question of the slice tower for $\Si{2+j\rho_K} H_Q \ulZ$, given that 
$\Si{1+j\rho_K} H_Q \ulm \simeq \phi_Z^* ( \Si{1+j\rho_K} H_K \ulm )$ is an $8j+4$-slice \cite{Slone}*{Proposition~5.7}. 
We determine the slices of $\Si{r+j\rho_K} H_Q \ulZ$, for $r\in\{2,\dots,5\}$ in \cref{sec:phiZK4Zslicetowers}.

%%%%%%%%%%%%%%%%%%%%%%%%%%%%%%%
\subsection{Slice towers for $\Si{r+j\rho_K} H_{Q} \ulZ$}
\label{sec:phiZK4Zslicetowers}

The $K_4$-spectrum $\Si{r+j\rho_{K}}H_{K} \ulZ$ is an $n$-slice for $r\in \{2,\dots,5\}$ \cite{Slone}*{Proposition~7.1}. However, the inflation of this to $Q_8$ is no longer a slice. We here determine the slice towers of these inflations. Throughout, we will implicitly use \cref{n0to4slices}, which does not rely on the following material.

\subsubsection{($r=2$)} 
\label{r2case}

First, we observe that $\Si{2+\rho_K} H_{Q} \ulZ$ is a 6-slice. 
To see this 
we first note that it restricts to a 6-slice at every proper subgroup by \cref{nsliceSInZK}. It therefore remains only to show that it does not have any $8k$-slices for $k\geq 1$. This is equivalent to showing that $\mpi_{-2} \left( \Si{\rho_K - k\rho_Q} H_{Q} \ulZ \right)$ vanishes for $k\geq 1$. In the case $k=1$, \eqref{CohomS(H)} shows that $\Si{-\QH} H_{Q} \ulZ$ is $(-3)$-truncated, in the sense that it has no homotopy Mackey functors above dimension $-3$. This remains true after further desuspending by copies of $\rho_Q$.

Next, the tower for $\Si{2+2\rho_K} H_{Q} \ulZ$ is given by
\[
\begin{tikzcd}
    P^{14}_{14} = \Si{-1+2\rho_Q} H_{Q}  \ul{w}^* \ar[r] 
    & \Si{2+2\rho_K} H_{Q} \ulZ \ar[d] \\
    P^{12}_{12} = \Si{1+\rho_Q} H_{Q}  \ulm \ar[r] 
    & \Si{2+\rho_Q} H_{Q} \ulZ(2,0) \ar[d] \\
    & P^{10}_{10} = \Si{2+\rho_{Q}} H_{Q}\ulZ(1,0).
\end{tikzcd}
\]
This uses the computation (see \cref{AuxZ20})
\[ \mpi_n \left( \Si{\QH-\rho_K} H_{Q} \ulZ(2,0) \right) \iso 
\begin{cases}
 \ulZ & n=0 \\
\ul{w}^* & n=-2
\end{cases}
\]
to produce the first cofiber sequence.

Finally, for $j\geq 3$, the tower may be obtained by recursively using 
\[
\begin{tikzcd}
    P^{8j-2}_{8j-2} = \Si{-1+j\rho_Q} H_{Q}  \ul{w}^* \ar[r] 
    & \Si{2+j\rho_K} H_{Q} \ulZ \ar[d] \\
    P^{8j-4}_{8j-4} = \Si{1+(j-1)\rho_Q} H_{Q}  \ulm \ar[r] 
    & \Si{2+(j-2)\rho_K + \rho_Q} H_{Q} \ulZ(2,0) \ar[d] \\
    P^{8j-6}_{8j-6} = \Si{1+(j-1)\rho_Q}H_{Q} \phi_Z^* \ulF \ar[r]
    &  \Si{2+(j-2)\rho_K + \rho_{Q}} H_{Q}\ulZ(1,0) \ar[d] \\
    & \Si{2+(j-2)\rho_K+\rho_Q} H_{Q} \ulZ.
\end{tikzcd}
\]
We have proved the following result.

\begin{prop}
\label{r2slices}
Let $j\geq 1$. The bottom slice of $\Si{2+j\rho_K} H_Q\ulZ$ is
\[
P^{2+4j}_{2+4j} \left( \Si{2+j\rho_K} H_{Q} \ulZ \right) \simeq
\begin{cases}
\Si{1 + \rho_K +  \frac{j-1}2\rho_Q} H_{Q} \ulZ^* & j\ \text{odd} \\
\Si{2+\frac{j}2\rho_Q} H_{Q} \ulZ & j\ \text{even}. \\
\end{cases}
\]
\end{prop}

\subsubsection{($r=3$)} 
\label{r3case}

By \eqref{CohomS(H)}, the cohomology of $S^\QH$ is given by
\[
\widetilde{\ul{\rH}}^n ( S^\QH; \ulZ) \iso
\mpi_{-n} \left( \Si{-\QH}  H_Q \ulZ  \right) \iso
\begin{cases}
\ulZ^* & n = 4 \\
\mystery & n = 3. \\
\end{cases}
\]
Suspending by $3+\rho_Q$ leads to the cofiber sequence
\begin{equation*}
\begin{tikzcd}
    P^{8}_{8} =  \Si{\rho_{Q}} H_{Q} \mystery \ar[r] 
    & \Si{3+\rho_K} H_{Q}\ulZ \ar[d] \\
    & P^7_7 = \Si{\rho_{Q}-1} H_{Q}\ulZ^*.
\end{tikzcd}
\end{equation*}
The tower for $\Si{3+j\rho_K} H_{Q}\ulZ$, where $j\geq 2$, is then given recursively by

\begin{equation*}
\begin{tikzcd}
    P^{8j}_{8j} =  \Si{j\rho_{Q}} H_{Q} \mystery \ar[r] 
    & \Si{3+j\rho_K} H_{Q}\ulZ \ar[d] \\
    & \Si{(j-1)\rho_K + \rho_{Q}-1} H_{Q}\ulZ^* \ar[d,equal] \\
    P^{8j-4}_{8j-4} = \Si{2+(j-1)\rho_Q} H_{Q} \phi_Z^* \ulF \ar[r]
    &  \Si{3+(j-2)\rho_K + \rho_Q} H_{Q} \ulZ(1,0) \ar[d] \\
    & \Si{3+(j-2)\rho_K +\rho_Q} H_{Q} \ulZ.
\end{tikzcd}
\end{equation*}
The last cofiber sequence arises from \cref{MackeySESExamples}.
We have proved the following result.

\begin{prop}
\label{r3slices}
Let $j\geq 1$. The bottom slice of $\Si{3+j\rho_K} H_Q\ulZ$ is
\[
P^{3+4j}_{3+4j} \left( \Si{3+j\rho_K} H_{Q} \ulZ \right) \simeq
\begin{cases}
\Si{-1+ \frac{j+1}2\rho_Q} H_{Q} \ulZ^* & j\ \text{odd} \\
\Si{3+\frac{j}2\rho_Q} H_{Q} \ulZ & j\ \text{even}. \\
\end{cases}
\]
\end{prop}

\subsubsection{($r=4$)} 
\label{r4case}

The tower for $\Si{4+\rho_K} H_{Q} \ulZ$ is given by
\begin{equation*}
\begin{tikzcd}
    P^{12}_{12} = \Si{\rho_Q+1} H_{Q}  \ul{mg} \ar[r] 
    & \Si{4+\rho_K} H_{Q} \ulZ \simeq \Si{2\rho_K} H_{Q} \ulZ(3,1) \ar[d] \\
    P^{10}_{10} = \Si{\rho_Q+1}  \ul{w} \ar[r] 
    & \Si{2\rho_K} H_{Q} \ulZ(3,2) \ar[d] \\
    & P^8_8 = \Si{\rho_{Q}} H_{Q}\ulZ^*.
\end{tikzcd}
\end{equation*}
This uses the short exact sequence (\cref{MackeySESExamples})
\[ \ulZ(3,1) \into \ulZ(3,2) \onto \ulm^*, 
\]
the equivalence $\Si{\rho_K} H_K \ulm^* \simeq \Si2 H_K \ul{mg}$ (\cite{GY}*{Proposition~4.8}),
and the computation (see \cref{AuxZ32})
\[ \mpi_n \left( \Si{\rho_K-\QH} H_{Q} \ulZ(3,2) \right) \iso 
\begin{cases}
 \ul{w} & n=1 \\
\ulZ^* & n=0.
\end{cases}
\]
The tower for $\Si{4+j\rho_K} H_{Q}\ulZ$, where $j\geq 2$, may then be obtained recursively from
\begin{equation*}
\begin{tikzcd}
    P^{8j+4}_{8j+4} = \Si{1+j\rho_Q} H_{Q}  \ul{mg} \ar[r] 
    & \Si{4+j\rho_K} H_{Q} \ulZ \simeq \Si{(j+1)\rho_K} H_{Q} \ulZ(3,1) \ar[d] \\
    P^{8j+2}_{8j+2} = \Si{1+j\rho_Q}  \ul{w} \ar[r] 
    & \Si{(j+1)\rho_K} H_{Q} \ulZ(3,2) \ar[d] \\
    & \Si{(j-1)\rho_K+\rho_{Q}} H_{Q}\ulZ^* \ar[d,equal] \\
    P^{8j-2}_{8j-2} = \Si{3+(j-1)\rho_Q} H_Q \phi_Z^* \ulF \ar[r] 
    &   \Si{4+(j-2)\rho_K+\rho_Q} H_Q \ulZ(1,0) \ar[d] \\
     &  \Si{4+(j-2)\rho_K+\rho_Q} H_Q \ulZ.
\end{tikzcd}
\end{equation*}

\begin{prop}
\label{r4slices}
Let $j\geq 1$. The bottom slice of $\Si{4+j\rho_K} H_Q\ulZ$ is
\[
P^{4+4j}_{4+4j} \left( \Si{4+j\rho_K} H_{Q} \ulZ \right) \simeq
\begin{cases}
\Si{\frac{j+1}2\rho_Q} H_{Q} \ulZ^* & j\ \text{odd} \\
\Si{4+\frac{j}2\rho_Q} H_{Q} \ulZ & j\ \text{even}. \\
\end{cases}
\]
\end{prop}

\subsubsection{($r=5$)}
\label{r5case}

Here, we start with the slice tower for $\Si5 H_Q \ulZ$, as this is not a slice. 
The short exact sequence
\[ \ulZ(3,1) \into \ulZ \onto \phi_Z^* \ul B(2,0) \]
gives rise to a cofiber sequence
\[
 P^8_8 = \Si{\rho_Q} H_{Q} \phi_Z^* \ul B(2,0) \rtarr
    \Si{5} H_{Q} \ulZ \simeq \Si{1+\rho_K} H_{Q} \ulZ(3,1)
    \rtarr 
    \Si{1+\rho_K} H_{Q}\ulZ.
\]
Now the argument showing that $\Si{2+\rho_K} H_Q \ulZ$ is a 6-slice, given above in \cref{r2case},
also applies to show that $\Si{1+\rho_K} H_Q \ulZ$ is a 5-slice. Thus, this cofiber sequence is the slice tower for $\Si5 H_Q \ulZ$.

Next, the tower for $\Si{5+\rho_K} H_{Q} \ulZ$ is  given by
\begin{equation*}
\begin{tikzcd}
    P^{16}_{16} = \Si{2\rho_Q} H_{Q} \phi_Z^* \ul B(2,0) \ar[r] 
    & \Si{5+\rho_K} H_{Q} \ulZ \simeq \Si{1+2\rho_K} H_{Q} \ulZ(3,1) \ar[d] \\
    P^{12}_{12} = \Si{2+\rho_Q} H_{Q} \phi_Z^* \ulF \ar[r] 
    & \Si{1+2\rho_K} H_{Q}\ulZ \ar[d] \\
    & P^9_9 = \Si{1+\rho_{Q}} H_{Q}\ulZ^*,
\end{tikzcd}
\end{equation*}
where the bottom cofiber sequence arises from the computation (\cref{AuxZ})
\[ \mpi_n \left( \Si{\rho_K-\QH} H_{Q} \ulZ \right) \iso 
\begin{cases}
\phi_Z^* \ulF & n=1 \\
\ulZ^* & n=0.
\end{cases}
\]

The tower for $\Si{5+j\rho_K} H_{Q}\ulZ$, where $j\geq 2$, may then be obtained recursively from
\begin{equation*}
\begin{tikzcd}
    P^{8j+8}_{8j+8} = \Si{(j+1)\rho_Q} H_{Q} \phi_Z^* \ul B(2,0) \ar[r] 
    & \Si{5+j\rho_K} H_{Q} \ulZ \ar[d,equal] \\
    &  \Si{1+(j+1)\rho_K} H_{Q} \ulZ(3,1) \ar[d] \\
    P^{8j+4}_{8j+4} = \Si{2+j\rho_Q} H_{Q} \phi_Z^* \ulF \ar[r] 
    & \Si{1+(j+1)\rho_K} H_{Q}\ulZ \ar[d] \\
    P^{8j}_{8j} = \Si{j\rho_Q}H_Q \ul B(3,0) \ar[r]
    &  \Si{1+(j-1)\rho_K+\rho_{Q}} H_{Q}\ulZ^* \ar[d] \\
    & \Si{1+(j-1)\rho_K+\rho_{Q}} H_{Q}\ulZ.
\end{tikzcd}
\end{equation*}

\begin{prop}
\label{r5slices}
Let $j\geq 1$. The bottom slice of $\Si{5+j\rho_K} H_Q\ulZ$ is
\[
P^{5+4j}_{5+4j} \left( \Si{5+j\rho_K} H_{Q} \ulZ \right) \simeq
\begin{cases}
\Si{1+ \frac{j+1}2\rho_Q} H_{Q} \ulZ^* 
& j\ \text{odd} \\
\Si{1 + \rho_K + \frac{j}2\rho_Q} H_{Q} \ulZ 
& j\ \text{even}. \\
\end{cases}
\]
\end{prop}

%%%%%%%%%%%%%%%%%%%%%%%%%%%%%%%%%%%%%%%%%%%%%%%%%%%
\subsection{Slices of $\Si{n} H_{Q} \ulZ$}
\label{sec:Q8Zslices}

In this section, we describe all slices of $\Sigma^n H_{Q}\ulZ$ for $n\geq 0$.

\begin{prop}
\label{n0to4slices}
The $Q_8$-spectrum $\Sigma^n H_{Q}\ulZ$ is an $n$-slice for $0\leq n \leq 4$.
\end{prop}

\begin{pf}
Since this is true after restricting to any $C_4$ (see \cref{sec:C4slice}), 
any higher slices would necessarily be geometric and therefore occurring in slice dimension at least 8. But we can show directly that $\Sigma^n H_Q \ulZ < 8$ if $n\in [0,4]$. This follows from the vanishing of $\pi_{\rho_Q}\Sigma^n H_Q \ulZ\iso \pi_{-n} \Si{-\rho_Q} H_Q \ulZ$ as displayed in 
\cref{fig:HtpySigmarhoQZ}.
\end{pf}

It remains to determine the slices of $\Sigma^n H_{Q}\ulZ$ when $n\geq 5$.
Note that
\cref{MainSliceInflationTheorem} applies by \cite{Slone}*{Proposition~8.5}.
We first describe the bottom slice.

\begin{prop}[The $n$-slice]
\label{nsliceZQ8}
For $n\geq 5$, write $n=8k+r$, where $r\in[5,12]$.
Then the $n$-slice of $\Si{n} H_{Q} \ulZ$ is
\[
P^n_n \left( \Si{n} H_{Q} \ulZ \right) \simeq
\begin{cases}
\Si{1+\rho_K +k\rho_Q} H_{Q} \ulZ & r=5 \\
\Si{2+\rho_K +k\rho_Q} H_{Q} \ulZ(3,2) & r=6 \\
\Si{-1+ (k+1)\rho_Q} H_{Q} \ulZ^* & r=7 \\
\Si{(k+1)\rho_Q} H_{Q} \ulZ^* & r=8 \\
\Si{1+(k+1)\rho_Q} H_{Q} \ulZ^* & r=9 \\
\Si{2+(k+1)\rho_Q} H_{Q} \ulZ(1,0) & r=10 \\
\Si{3+(k+1)\rho_Q} H_{Q} \ulZ & r=11 \\
\Si{4+(k+1)\rho_Q} H_{Q} \ulZ & r=12. \\
\end{cases}
\]
\end{prop}

\begin{pf}
By \cref{MainSliceInflationTheorem}, the $n$-slice of $\Si{n} H_{Q} \ulZ$ is the $n$-slice of the $\ulZ$-module inflation of the $n$-slice of $\Si{n} H_{K} \ulZ$. 
By \cite{Slone}*{Proposition~8.5}, writing $n=4j+r_4$ with $r_4\in \{2,3,4,5\}$, we have
\[
\HphiZ_Z^* P^n_n \left( \Sigma^n H_{K_4} \ulZ \right) \simeq
\begin{cases}
\Si{2+j\rho_K} H_Q \ulZ(2,1) & n\equiv 2 \pmod4 \\
\Si{r_4+j\rho_K} H_Q \ulZ & \text{else}.
\end{cases}
\]
If $n\not\equiv 2 \pmod4$, the slice tower was given in \cref{sec:phiZK4Zslicetowers}.
For the case of $n\equiv2$, 
since  
$\Si{1+j\rho_K} H_Q \ulm \simeq \phi_Z^* ( \Si{1+j\rho_K} H_K \ulm )$ is an $8j+4$-slice \cite{Slone}*{Proposition~5.7},
the cofiber sequence (\cref{MackeySESExamples})
\begin{equation}
\label{Z21reduction}
\Si{1+j\rho_K} H_Q \ulm \rtarr \Si{2+j\rho_K} H_Q \ulZ(2,1) \rtarr \Si{2+j\rho_K} H_Q \ulZ,
\end{equation}
combines with the work of \cref{r2case} to 
to show that 
\[
P^n_n \left( \Si{2+j\rho_K} H_Q \ulZ(2,1) \right) \simeq P^n_n \left( \Si{2+j\rho_K} H_Q \ulZ\right). 
\]
The latter is given in \cref{r2slices}.
\end{pf}

\begin{prop}[The $8k$-slices]
\label{8kslicesZQ8}
For $n\geq 5$ and $8k>n$, the $8k$-slice of $\Si{n} H_{Q} \ulZ$ is
\[
P^{8k}_{8k}\left( \Si{n} H_{Q} \ulZ \right) \simeq
\begin{cases}
\Si{k} H_{Q} \ulg^{n-k-3} & 8k \in [4n-8,8n-32] \\
\Si{k\rho_Q} H_{Q} \ulg^{\frac{4k-n}2} & 8k \in [2n+4,4n-16]\\
 & \quad \text{and} \quad n\equiv0\pmod2 \\
\Si{k\rho_Q} H_{Q} \ulg^{\frac{4k-n-3}2}\oplus \phi_{LDR}^* \ulF^* & 8k \in [2n+4,4n-12]  \\
 & \quad \text{and} \quad n\equiv1\pmod2 \\
\Si{k\rho_Q} H_{Q} \ul{mg}^* & 8k=2n+2 \\
\Si{k\rho_Q} H_{Q} \phi_Z^* \ul B(2,0) & 8k = 2n-2 \\
\Si{k\rho_Q} H_{Q} \ul B(3,0) & 8k \in [n+3,2n-10] \\
 & \quad \text{and} \quad  n\equiv1\pmod4 \\
\Si{k\rho_Q} H_{Q} \mystery & 8k \in [n+1,2n] \\
 & \quad \text{and} \quad n\equiv3\pmod4. \\
\end{cases}
\]
\end{prop}

\begin{pf}
This is a translation of \cref{HomologySnegkrhoQ}.
Alternatively, the slices above dimension $2n$ follow from \cref{MainSliceInflationTheorem} and \cite{Slone}*{Proposition~8.6}. The slices in dimensions $2n$ and lower follow from the towers computed in \cref{sec:phiZK4Zslicetowers}.
\end{pf}

\begin{prop}[The $8k+4$-slices]
\label{8kplus4slicesZQ8}
For $n\geq 5$ and $8k+4 > n$, the $8k+4$-slices of $\Si{n}H_{Q} \ulZ$ are
\[
P^{8k+4}_{8k+4} \left( \Si{n} H_{Q} \ulZ \right) \simeq
\begin{cases}
\Si{3+k\rho_Q} H_{Q} \phi_{LDR}^* \ulF & 8k+4 \in [2n+4, 4n-12], \quad n\ \text{even} \\
\Si{2+k\rho_Q} H_{Q} \phi_Z^* \ulF & 8k+4 \in [n+1,2n-4], \quad n \ \text{odd} \\
\Si{1+k\rho_Q} H_{Q} \ulm & 8k+4 \in [n+2, 2n], \quad n \equiv 2 \pmod4 \\
\Si{1+k\rho_Q} H_{Q} \ul{mg} & 8k+4 \in [n+4, 2n-4], \quad n \equiv 0 \pmod4 \\
\end{cases}
\]
\end{prop}

\begin{pf}
The first case follows from \cite{Slone}*{Proposition~8.7}.
The remaining cases follow from \cref{Z21reduction} and \cref{sec:phiZK4Zslicetowers}.
\end{pf}

\begin{prop}[The $4k+2$-slices]
Let $n\geq 5$.
If $n$ is odd, then 
$\Si{n}H_{Q} \ulZ$ 
has no nontrivial $4k+2$-slices if 
$4k+2>n$. 
If $n$ is even and $8k+2 > n$, 
then the $8k+2$-slice of $\Si{n}H_{Q} \ulZ$  is nontrivial only if 
$8k+2 \in [n+1,2n]$, in which case the slice is
\[
P^{8k+2}_{8k+2} \left( \Si{n} H_{Q} \ulZ \right) \simeq
\begin{cases}
\Si{1+k\rho_Q} H_{Q} \ulw & n \equiv 0 \pmod4 \\
\Si{1+k\rho_Q} H_{Q} \phi_Z^* \ulF & n \equiv 2 \pmod4 \\
\end{cases}
\]
Similarly, if $n$ is even and $8k-2 > n$, the $8k-2$-slice is nontrivial only if
$8k-2 \in [n+1,2n]$, 
in which case the slice is
\[
P^{8k-2}_{8k-2} \left( \Si{n} H_{Q} \ulZ \right) \simeq
\begin{cases}
\Si{-1+k\rho_Q} H_{Q} \phi_Z^*\ulF^* & n \equiv 0 \pmod4 \\
\Si{-1+k\rho_Q} H_{Q} \ulw^* & n \equiv 2 \pmod4.  \\
\end{cases}
\]
\end{prop}

\begin{pf}
According to \cite{Slone}, the $K_4$-spectrum $\Si{n} H_K \ulZ$ does not have any nontrivial slices in odd dimensions, except for the $n$-slice. By \cref{MainSliceInflationTheorem}, this implies that $\Si{n} H_Q \ulZ$ does not have any $4k+2$-slices above dimension $2n$. The slices in dimensions below $2n$ are given by \cref{sec:phiZK4Zslicetowers}.
\end{pf}

%%%%%%%%%%%%%%%%%%%%%%%%%%%%%%%%%%%%%%%%%%%%%%%%%%%
\subsection{Slice towers for $\Si{n} H_{Q} \ulZ$}
\label{sec:Q8Zslicetowers}

By \cref{n0to4slices}, $\Si{n}H_{Q} \ulZ$ is an $n$-slice for $n\in \{0,\dots,4\}$.
The slice tower for $\Si5 H_Q \ulZ$ was given in \cref{r5case}. We now display a few more examples of slice towers.

\begin{eg}
\label{Si6ZQtower}
The slice tower for $\Si6 H_{Q}\ulZ$ is 
\begin{equation*}
\begin{tikzcd}
    P^{16}_{16} = \Si{2} H_{Q} \ulg \ar[r] 
    & \Si6 H_{Q}\ulZ \ar[d] \\
    P^{12}_{12} = \Si{1+\rho} H_{Q} \ulm \ar[r] 
    & \Si{2+\rho_K} H_{Q}\ulZ(2,1) \ar[d] \\
    & P^6_6 = \Si{2+\rho_K} H_{Q}\ulZ.
\end{tikzcd}
\end{equation*}
This follows immediately from combining \cite{Slone}*{Example~8.2}, \cref{Z21reduction}, and \cref{r2case}.
\end{eg}

\begin{eg}
The slice tower for $\Si7 H_{Q}\ulZ$ is 
\begin{equation*}
\begin{tikzcd}
    P^{24}_{24} = \Si{3} H_{Q} \ulg \ar[r] 
    & \Si7 H_{Q}\ulZ \ar[d] \\
    P^{16}_{16} = \Si{2+\rho_{Q}} H_{Q}  \ulm \ar[r] 
    & \Si{3+\rho_K} H_{Q}\ulZ(2,1) \ar[d] \\
    P^{8}_{8} =  \Si{\rho_{Q}} H_{Q} \mystery \ar[r] 
    & \Si{3+\rho_K} H_{Q}\ulZ \ar[d] \\
    & P^7_7 = \Si{\rho_{Q}-1} H_{Q_8}\ulZ^*.
\end{tikzcd}
\end{equation*}
This follows immediately from combining \cite{Slone}*{Example~8.3} and \cref{r3case}.
\end{eg}

\begin{eg}
The slices, but not the slice tower, for $\Si8 H_K \ulZ$ were determined in \cite{Slone}*{Section 8}.
Let us denote by $F$ the fiber of the map $H_{Q} \ulZ \rtarr H_{Q} \phi_{LDR} \ulF$ induced by the map of $Q_8$-Mackey functors $\ulZ \rtarr \phi_{LDR} \ulF$ that is surjective at $L$, $D$, and $R$. 
Then the nontrivial homotopy Mackey functors of $F$ are $\mpi_0(F) \simeq \ulZ(2,1)$ and $\mpi_{-1}(F) \iso \ulg^2$.
The slice tower for $\Si8 H_{Q}\ulZ$ is 
\begin{equation*}
\begin{tikzcd}
    P^{32}_{32} = \Si4 H_{Q} \ulg \ar[r] 
    & \Si8 H_{Q}\ulZ \simeq \Si{4+\rho_K} H_Q \ulZ(3,1)
    \ar[d] \\
    P^{24}_{24} = \Si3 H_{Q} \ulg^2 \ar[r] 
    & \Si{4+\rho_K} H_{Q}\ulZ(2,1) \ar[d] \\
    P^{20}_{20} = \Si{3+\rho_Q}H_{Q} \phi_{LDR}^*\ulF \ar[r] 
    & \Si{4+\rho_K} F \ar[d] \\
    P^{12}_{12} = \Si{1 + \rho_Q} H_{Q}  \ul{mg} \ar[r] 
    & \Si{4+\rho_K} H_{Q} \ulZ \simeq \Si{2\rho_K} H_{Q} \ulZ(3,1) \ar[d] \\
    P^{10}_{10} = \Si{\rho_Q+1}  \ul{w} \ar[r] 
    & \Si{2\rho_K} H_{Q} \ulZ(3,2) \ar[d] \\
    & P^8_8 = \Si{\rho_{Q}} H_{Q}\ulZ^*,
\end{tikzcd}
\end{equation*}
where the bottom of the tower comes from \cref{r4case}.
\end{eg}

\section{Homology calculations}
\label{sec:Homology}

In \cref{sec:Q8slice}, we described the slices of $\Si{n}H_Q \ulZ$. In \cref{sec:SpectralSequences} below, we will give the corresponding slice spectral sequences. The $E_2$-pages of those spectral sequences are given by the homotopy Mackey functors of the slices. We describe those homotopy Mackey functors here.

\subsection{The $n$-slice}

We start with the \(n\)-slices in the order listed in \cref{nsliceZQ8}. The homotopy Mackey functors of \(\Si{j\rho_Q} H_Q\ulZ\) were calculated in \cref{HomologySkrhoQ}. We use the same methods to determine the homotopy Mackey functors of \(\Si{\rho_K + j\rho_Q} H_Q\ulZ\).

\begin{prop}
\label{Htpynslicer5}
For \(j\geq 1\), the  homotopy Mackey functors of \(\Si{\rho_K+j\rho_Q} H_Q\ulZ\) are
\begin{align*}
    \ul\pi_{i}( \Si{\rho_K + j\rho_Q} H_Q\ulZ) &\cong \left\lbrace \begin{array}{ll}
        \ulZ & i=8j+4 \\
        \mystery & \begin{aligned} i\in[4j+4,8j+3], \\ i\equiv 2\pmod 4 \end{aligned} \\
        \ul B(3,0) & \begin{aligned} i\in[4j+4,8j+3], \\ i\equiv 0\pmod 4 \end{aligned} \\
        \phi^*_Z \ul\pi_{i}(\Si{(j+1)\rho_K} H_K\ulZ) & i\in[j+1,4j+3].
    \end{array}\right.
\end{align*}
\end{prop}

See \cref{HomologySkrhoK} or \cref{fig:HtpySigmarhoK4Z} for the homotopy Mackey functors of $\Si{(j+1)\rho_K} H_K\ulZ$.

\begin{comment}
Next, we use the equivalence \(\Psi^*_Z H_K\ulZ(2,1) \simeq H_Q\ulZ(3,2)\) to compute the homotopy of \(\Si{\rho_K+j\rho_Q} H_Q\ulZ(3,2)\). We will need the following.

\begin{prop}
For \(j\geq 1\), the homotopy of \(\Si{j\rho_K} H_Q\ulZ(3,2)\) is
\begin{align*}
    \ul\pi_{i}( \Si{j\rho_K} H_Q\ulZ(3,2)) 
    &\cong \Psi^*_Z \ul\pi_i (\Si{j\rho_K} H_K\ulZ(2,1)) \\
    &\cong \left\lbrace \begin{array}{ll}
        \ulZ & i=4j \\
        \ul{mg} & i=4j-2 \\
        \ulg^{\tfrac{1}{2}(4j-i-1)}  & i\in[2j,4j-3],\ i\ \text{odd} \\
        \ulg^{\tfrac{1}{2}(4j-n+2)-3}\oplus \phi^*_{LDR}\ulF &  i\in[2j,4j-3],\ i\ \text{even} \\
        \ulg^{i-j+1} & i\in[j+1,2j-1].
    \end{array}\right.
\end{align*}
\end{prop}

\begin{proof}
The Klein-4 short exact sequence \(\ulZ(2,1) \rtarr \ulZ \rtarr \ulg\) and \cite[Proposition 9.1]{Slone} provide the homotopy of \(\Si{j\rho_K} H_K\ulZ(2,1)\). The result then follows from  \cref{HtpyInflation}.
\end{proof}

The sequence
\begin{align*}
    S(j\mathbb{H})_+ \smsh \Si{(j+1)\rho_K} H_Q\ulZ(3,2) \rtarr \Si{(j+1)\rho_K} H_Q\ulZ(3,2) \rtarr \Si{\rho_K + j\rho_Q} H_Q\ulZ(3,2)
\end{align*}
then gives us the following.

\bnote{We can get this result from \cref{Htpynslicer5} by using $\ulZ(3,2) \into \ulZ \to \ulg$
}
\cnote{Should we use that sequence to get 7.2 and 7.3 then? Also, did we determine that the homotopy of the inflation is the inflation of the homotopy?}
\bnote{Yes, this is now \cref{HtpyInflation}}
\end{comment}

We may now use \cref{Htpynslicer5} and the exact sequence \(\ulZ(3,2) \hookrightarrow \ulZ \twoheadrightarrow \ulg\) to get the homotopy Mackey functors of \(\Si{\rho_K+j\rho_Q} H_Q\ulZ(3,2)\).

\begin{prop}
For \(j\geq 1\), the  homotopy Mackey functors of \(\Si{\rho_K+j\rho_Q} H_Q\ulZ(3,2)\) are
\begin{align*}
    \ul\pi_{i}( \Si{\rho_K + j\rho_Q} H_Q\ulZ(3,2)) &\cong \left\lbrace \begin{array}{ll}
        \ulZ & i=8j+4 \\
        \mystery & \begin{aligned} i\in[4j+4,8j+3], \\ i\equiv 2\pmod 4 \end{aligned} \\
        \ul B(3,0) & \begin{aligned} i\in[4j+4,8j+3], \\ i\equiv 0\pmod 4 \end{aligned} \\
        \phi^*_Z \ul\pi_{i}(\Si{(j+1)\rho_K} H_K\ulZ) & i\in[j+2,4j+3].
    \end{array}\right.
\end{align*}
\end{prop}

 The key point here is that the homotopy Mackey functors of \(\Si{\rho_K+j\rho_Q} H_Q\ulZ(3,2)\) are the same as that of \(\Si{\rho_K+j\rho_Q} H_Q\ulZ\), except that the \(\ulg\) in degree \(j+1\) has been removed.

In \cref{HomologySnegkrhoQ} we list the homotopy Mackey functors of \(\Si{-j\rho_Q} H_Q\ulZ\). Anderson duality then provides us with the homotopy Mackey functors of \(\Si{j\rho_Q} H_Q\ulZ^*\).

\begin{prop}
\label{HtpyZdual}
For \(j\geq 1\), the  homotopy Mackey functors of \(\Si{j\rho_Q} H_Q \ulZ^*\) are
\begin{align*}
    \ul\pi_i( \Si{j\rho_Q} H_Q\ulZ^*) &\cong \left\lbrace \begin{array}{ll}
        \ulZ & i=8j \\
        \mystery & \begin{aligned} i\in[4j+1, 8j-1], \\ i\equiv 2\mod 4 \end{aligned} \\
        \ul B(3,0) & \begin{aligned} i\in[4j+1,8j-1], \\ i\equiv 0\mod 4\end{aligned} \\
        \phi^*_Z\ul B(2,0) & i=4j \\
        \phi_Z^* \mpi_{i-4} (\Si{(j-1)\rho_K} H_K \ulZ) & i\in[j+3,4j-1].
    \end{array}\right.
\end{align*}
\end{prop}

Finally, the homotopy Mackey functors of \(\Si{j\rho_Q} H_Q\ulZ(1,0)\) follow from the exact sequence \(\ulZ(1,0) \hookrightarrow \ulZ \twoheadrightarrow \phi^*_Z\ulF\).

\begin{prop}
For \(j\geq 1\), the  homotopy Mackey functors of \(\Si{j\rho_Q} H_Q\ulZ(1,0)\) are
\begin{align*}
    \ul\pi_{i}( \Si{j\rho_Q} H_Q\ulZ(1,0)) &\cong \left\lbrace \begin{array}{ll}
        \ulZ & i=8j \\
        \mystery & \begin{aligned} i\in[4j+1,8j-2], \\ i\equiv 2\pmod 4 \end{aligned} \\
        \ul B(3,0) & \begin{aligned} i\in[4j+1,8j-2], \\ i\equiv 0\pmod 4 \end{aligned} \\
        \phi^*_Z\ul B(2,0) & i=4j \\
        \phi_Z^* \mpi_i (\Si{j\rho_K} H_K \ulZ) & i\in[j,4j-1].
    \end{array}\right.
\end{align*}
\end{prop}

\begin{comment}
\bnote{We can get this more easily from the cofiber sequence $\ulZ(1,0) \rtarr \ulZ \rtarr \phi_Z^* \ulF$.}

\begin{proof}
As \(\Si{j\rho_Q} H_Q\ulZ(1,0) \simeq I_{\Z} \Si{-j\rho_Q} H_Q\ulZ(3,1)\), the result follows by taking the Anderson dual of the homotopy listed in \cref{HomologySnegkrhoQZ31}. Note the dimension shift for the torsion elements. This occurs because for a torsion spectrum \(X\), its Anderson dual is a shifted copy of its Brown-Comenetz dual; that is, \(I_{\Q/\Z} X \simeq \Si{1} I_{\Z}X\).
\end{proof}
\end{comment}

\subsection{The $8k$-slices}
We now move on to the  \(8k\)-slices.

\begin{prop}
\label{HtypjrhophiB20}
For \(j=1\), the homotopy Mackey functors of \(\Si{j\rho_Q} H_Q\phi^*_Z \ul B(2,0)\) are
\begin{align*}
    \ul\pi_i( \Si{\rho_Q} H_Q\phi^*_Z\ul B(2,0)) &\cong \left\lbrace \begin{array}{ll}
        \ul{mg} & i=2 \\
        \ulg & i=1.
    \end{array}\right.
\end{align*}
For \(j\geq 2\), they are
\begin{align*}
    \ul\pi_i( \Si{j\rho_Q} H_Q\phi^*_Z\ul B(2,0)) &\cong \left\lbrace \begin{array}{ll}
        \phi^*_{LDR}\ulF & i=2j \\
        \ulg^3 & i\in[j+2,2j-1] \\
        \ulg^2 & i=j+1 \\
        \ulg & i=j.
    \end{array}\right.
\end{align*}
\end{prop}

\begin{proof}
Because \(\phi^*_Z\ul B(2,0)\) is a pullback,
\begin{align*}
    \Si{j\rho_Q} H_Q\phi^*_Z \ul B(2,0) \simeq \Si{j\rho_K} H_Q\phi^*_Z \ul B(2,0).
\end{align*}
The exact sequence of \(K\)-Mackey functors \(\ul m^*\rtarr \ul B(2,0) \rtarr \ulg\) provides us with \(\Si{j\rho_K} H_K\ul m^* \rtarr \Si{j\rho_K} H_K\ul B(2,0) \rtarr \Si{j\rho_K} H_K\ulg\). The conclusion follows from \cite{GY}*{Propositions~4.8 and 7.4} and the resulting long exact sequence in homotopy.
\end{proof}

We may again use this strategy of reducing the calculations from \(Q\) to \(K\) for determining the homotopy Mackey functors of \(\Si{j\rho_Q} H_Q\ul B(3,0)\).

\begin{prop} \label{HtpySigmakRhoHQM}
For \(j=1\) the homotopy Mackey functors of \(\Si{j\rho_Q} H_Q\ul B(3,0)\) are
\begin{align*}
    \ul\pi_i( \Si{\rho_K} H_K\ul B(3,0)) &\cong \left\lbrace \begin{array}{ll}
        \phi^*_Z\ulF & i=4 \\
        \ul{mg} & i=2 \\
        \ulg & i=1.
    \end{array}\right.
\end{align*}
For \(j\geq 2\), the homotopy Mackey functors of \(\Si{j\rho_Q} H_Q \ul B(3,0)\) are
\begin{align*}
    \ul\pi_i( \Si{j\rho_Q} H_Q\ul B(3,0)) &\cong \left\lbrace \begin{array}{ll}
        \phi^*_Z\ulF & i=4j \\
        \ul{mg} & i=4j-1 \\
        \phi^*_{LDR}\ulF\oplus g^{4j-2-i} & i\in [2j+2,4j-2] \\
        \ulg^{2(k-2)+1} & i=2j+1 \\
        \phi^*_{LDR}\ulF\oplus \ulg^{2(j-3)+1} & i=2j \\
        \ulg^{2(i-j-1)} & i\in[j+3,2j-1] \\
        \ulg^{i-j+1} & i\in[j,j+2].
    \end{array}\right.
\end{align*}
\end{prop}

\begin{proof}
Because the underlying spectrum of \( H_Q\ul B(3,0)\) is contractible,
\begin{align*}
    \Si{\rho_Q} H_Q\ul B(3,0) \simeq \Si{\rho_K} H_Q\ul B(3,0).
\end{align*}
Now, we may consider \(\ul B(3,0)\) as a pullback \(\phi^*_Z B:= \ul B(3,0)\), thus the calculation is reduced to one of $K$-Mackey functors. The sequence of \(K\)-Mackey functors \(\ulZ^*\xrtarr{2}\ulZ \rtarr \ul B\) provides us with
\begin{align*}
    \Si{j\rho_K} H_K\ulZ^* \rtarr \Si{j\rho_K} H_K\ulZ \rtarr \Si{j\rho_K} H_K\ul B.
\end{align*}
Except for \(i=4j-2\), the result follows from the associated long exact sequence in homotopy. In degree \(4j-2\) we have an extension 
\begin{align*}
    \ul{mg} \rtarr \ul\pi_{4j-2}(\Si{j\rho_K} H\ul B) \rtarr \ulg.
\end{align*}
We need to show this is not the split extension. This follows from the exact sequence \(\ul B(2,0)\rtarr \ul B \rtarr \ulF\) of $K$-Mackey functors.
\end{proof}

\begin{prop}
\label{Htpyjrhomgw}
For \(j=1\) and \(j=2\), the homotopy Mackey functors of \(\Si{j\rho_Q} H_Q \mystery \) are
\[
\ul\pi_i( \Si{\rho_Q} H_Q \mystery) \cong
    \begin{cases}
    \phi_Z^* \ulF & i=4 \\
    \phi_Z^* \ul{B}(2,0) & i=2.
    \end{cases}
\]
and
\[
\ul\pi_i( \Si{2\rho_Q} H_Q \mystery) \cong
    \begin{cases}
    \phi_Z^* \ulF & i=8 \\
    \ul{mg} & i=7 \\
    \phi_{LDR} \ulF & i=6 \\
    \ulg & i=5 \\
    \ul{mg} & i=4 \\
    \ulg & i=3.
    \end{cases}
\]
For \(j\geq 3\), the homotopy Mackey functors of \(\Si{j\rho_Q} H_Q \mystery\) are
\[
\mpi_i ( \Si{j\rho_Q} H_Q \mystery ) \iso 
\begin{cases}
\phi_Z^* \ulF & i=4j \\
\ul{mg} & i=4j-1 \\
\phi_{LDR} \ulF \oplus \ulg^{4j-i-2} & i \in [2j+2,4j-2] \\
\ulg^{2j-3} & i=2j+1 \\
\ulg^{2j-5} \oplus \phi_{LDR} \ulF & i=2j \\
\ulg^{2(i-j)-2} & i \in [j+2,2j-1] \\
\ulg & i=j+1
\end{cases}
\]
\end{prop}

\begin{pf}
We first deal with the case $j=1$. The short exact sequence of Mackey functors
\[ \ul{w}^* \into \mystery \onto \ul{mg}^*\]
combines with \cref{HomologySkrhoQwstar} and \cref{Htpyjrhomgstar} to show that the only nontrivial Mackey functors are $\phi_Z^* \ulF$ in degree 4 and an extension of $\ulm$ by $\ulg$ in degree 2. It remains to see that this extension is $\phi_Z^* \ul{B}(2,0)$.
According to \cref{HomologySnegkrhoQ}, the Postnikov tower for $\Si{-\rho_Q} H_Q \ulZ$ is
\[
\begin{tikzcd}
\Si{-5} H_Q \phi_Z^* \ul{B}(2,0) \ar[r] & \Si{-\rho_Q} H_Q \ulZ \ar[d] \\
\Si{-7} H_Q \mystery \ar[r] & X \ar[d] \\
 & \Si{-8} H_Q \ulZ^*.
\end{tikzcd}
\]
Desuspending this diagram once by $\rho_Q$
gives a tower for computing the homotopy Mackey functors of $\Si{-2\rho_Q} H_Q \ulZ$.
The homotopy Mackey functors for $\Si{-8-\rho_Q} H_Q \ulZ^*$ and $\Si{-5\rho_Q} H_Q \phiZ^* \ul{B}(2,0)$ follow, using Anderson duality, from   \cref{HomologySkrhoQ} and \cref{HtypjrhophiB20}.
Long exact sequences in homotopy then imply that 
\[\mpi_{-9} ( \Si{-7-\rho_Q} H_Q \mystery ) \iso \phi_Z^* \ul{B}(2,0).\]
Dualizing gives that $\mpi_2 ( \Si{\rho_Q} H_Q \mystery )$ is $\phi_Z^* \ul{B}(2,0)$.

We now have a fiber sequence
\begin{equation}
\label{DecompRHOmgw}
\Si4 H_Q \phi_Z^* \ulF \rtarr \Si{\rho_Q} H_Q \mystery \rtarr \Si2 H_Q \phi_Z^* \ul{B}(2,0).     
\end{equation} 
Suspending this sequence by $\rho_Q$ immediately gives the homotopy Mackey functors of $\Si{2\rho_Q} H_Q \mystery$.
The same is true in the case $j=3$, except that we have an extension
\[ \ulg \into \mpi_6 \Si{3\rho_Q} H_Q \mystery \onto \phi_{LDR} \ulF.
\]
We claim that, more generally, any extension of $\ulZ$-modules
\[ 
\ulg^m \into \mf{E} \onto \phi_{LDR}\ulF
\]
is necessarily the split extension. To see this, first note that $\phi_{LDR}\ulF$ is, by definition, the direct sum $\phi_L^* \ulF \oplus \phi_D^* \ulF \oplus \phi_R^* \ulF$. It therefore suffices to show that the only $\ulZ$-module extension of $\phi_L^* \ulF$ by $\ulg^m$ is the split extension. Since any such extension will vanish at the subgroups $D$ and $R$, the $\ulZ$-module structure forces the value at $Q$ to be 2-torsion and therefore equal to $\ulF^{m+1}$. Since there is a nontrivial restriction to the subgroup $L$, the $\ulZ$-module structure forces the transfer from $L$ to vanish. Thus the extension must be the split extension. 

The suspension by $(j-1)\rho_Q$ of \cref{DecompRHOmgw} gives the homotopy Mackey functors of $\Si{j\rho_Q}H_Q\mystery$ in degrees $2j+1$ and higher. 
Now we argue by induction that the Mackey functors for $\Si{j\rho_Q} H_Q \mystery$ are as claimed, for $j\geq 3$. 
For instance, since the bottom Mackey functor is
\[
\mpi_j ( \Si{(j-1)\rho_Q} H_Q \mystery ) \iso \ulg,
\]
we see by decomposing $\Si{(j-1)\rho_Q} H_Q \mystery$ using the Postnikov tower that \[\mpi_{j+1} ( \Si{j\rho_Q} H_Q \mystery ) \iso \ulg.\]
The values of the Mackey functors $\mpi_i$, for $i \leq 2j-2$, follow in a similar way. The values
\[
\mpi_{2j-2} ( \Si{(j-1)\rho_Q} H_Q \mystery ) \iso \ulg^{2j-7} \oplus \phi_{LDR} \ulF,\]
and
\[
\mpi_{2j-1} ( \Si{(j-1)\rho_Q} H_Q \mystery ) \iso \ulg^{2j-5} 
\]
give that
\[
\mpi_{2j-1} ( \Si{j\rho_Q} H_Q \mystery ) \iso \ulg^{2j-4} 
\]
and that we have an extension of $\ulZ$-modules
\[
\ulg^{2j-5} \into \mpi_{2j} ( \Si{j\rho_Q} H_Q \ulZ ) \onto \phi_{LDR} \ulF.
\]
By the argument given above, this must be the split extension.
\end{pf}

The homotopy Mackey functors for the remaining \(8k\)-slices 
follow from \cite[Propositions 9.5, 9.8]{Slone}. 

\begin{prop}[{\cite[Proposition 9.5]{Slone}}, {\cite[Proposition 4.8]{GY}}]
\label{Htpyjrhomgstar}
We have the equivalence \(\Si{\rho_Q} H_Q\ul{mg}^* \simeq \Si{2} H_Q\ul m\). For \(j\geq 2\), the homotopy Mackey functors of \(\Si{j\rho_Q} H_Q \ul{mg}^*\) are
\begin{align*}
    \ul\pi_i( \Si{j\rho_Q} H_Q\ul{mg}^*) &\cong \left\lbrace \begin{array}{ll}
        \phi^*_{LDR}\ulF & i=2j \\
        \ulg^3 & i\in[j+2,2j-1] \\
        \ulg & i=j+1.
    \end{array}\right.
\end{align*}
\end{prop}

\begin{prop}[{\cite[Proposition 9.8]{Slone}}]
We have  equivalences
\begin{align*}
\Si{j\rho_Q} H\phi^*_{LDR}\ulF^* \simeq \left\lbrace \begin{array}{ll}
	    \Si{2} H\phi^*_{LDR}\ul f & j=1 \\
	    \Si{4} H\phi^*_{LDR}\ulF & j=2.
	\end{array}\right.
\end{align*}
Then for \(j\geq 3\), the nontrivial homotopy Mackey functors of \(\Si{j\rho_Q} H\phi^*_{LDR} \ulF^*\) are
\begin{align*}
	\ul\pi_{i} ( \Si{j\rho_Q} H_Q\phi^*_{LDR} \ulF^* ) = \left\lbrace \begin{array}{ll}
	\phi^*_{LDR} \ulF & i=2j\\
	\ulg^{3}  & i\in [j+2,2j-1].
	\end{array}\right.
\end{align*}
\end{prop}

\subsection{The $8k+4$-slices}
Similarly, the homotopy Mackey functors of the \((8k+4)\)-slices follow from  \cite[Proposition 9.8]{Slone} and \cite[Corollary 7.2, Propositions 7.3, 7.4]{GY}.

\begin{prop}[{\cite[Proposition 3.6]{GY}}]
For \(j\geq 1\), the homotopy Mackey functors of \(\Si{j\rho_Q} H_Q\phi^*_{LDR} \ulF\) are
\begin{align*}
    \ul\pi_i( \Si{j\rho_Q} H_Q\phi^*_{LDR} \ulF) &\cong \left\lbrace \begin{array}{ll}
        \phi^*_{LDR}\ulF & i=2j \\
        \ulg^3 & i\in[j,2j-1].
    \end{array}\right.
\end{align*}
\end{prop}

\begin{prop}[{\cite[Corollary 7.2]{GY}}] \label{HtpyjrhoQphiF}
For \(j\geq 1\), the homotopy Mackey functors of \(\Si{j\rho_Q} H_Q \phi^*_Z\ulF\) are
\begin{align*}
    \ul\pi_i( \Si{j\rho_Q} H_Q\phi^*_Z \ulF) &\cong \left\lbrace \begin{array}{ll}
        \phi^*_Z\ulF & i=4j \\
        \ul{mg} & i=4j-1 \\
        \phi^*_{LDR}\ulF\oplus g^{4j-2-i} & i\in[2j,4j-2] \\
        \ulg^{2(i-j)+1} & i\in[j,2j-1].
    \end{array}\right.
\end{align*}
\end{prop}

\begin{prop}[{\cite[Proposition 7.3]{GY}}]
For \(j\geq 1\), the homotopy Mackey functors of \(\Si{j\rho_Q} H_Q \ul m\) are
\begin{align*}
    \ul\pi_i( \Si{j\rho_Q} H_Q\ul m) &\cong \left\lbrace \begin{array}{ll}
        \phi^*_{LDR}\ulF & i=2j \\
        \ulg^{3} & i\in[j+1,2j-1] \\
        \ulg & i=j.
    \end{array}\right.
\end{align*}
\end{prop}

\begin{prop}[{\cite[Proposition 7.4]{GY}}]
For \(j\geq 1\), the homotopy Mackey functors of \(\Si{j\rho_Q} H_Q \ul {mg}\) are
\begin{align*}
    \ul\pi_i( \Si{j\rho_Q} H_Q\ul {mg}) &\cong \left\lbrace \begin{array}{ll}
        \phi^*_{LDR}\ulF & i=2j \\
        \ulg^{3} & i\in[j+1,2j-1]. \\
        \ulg^2 & i=j.
    \end{array}\right.
\end{align*}
\end{prop}

\subsection{The $4k+2$-slices}

The homotopy Mackey functors of the $(4k+2)$-slice \linebreak $\Si{1+k\rho_Q} H_Q \phi_Z^* \ulF$ are given in \cref{HtpyjrhoQphiF}. The homotopy Mackey functors of the remaining \((4k+2)\)-slices are as follows.

\begin{prop}[{\cite[Proposition 4.8, Corollary 7.2]{GY}}] \label{HtpyjrhoQphiF*}
We have the equivalence \(\Si{\rho_Q} H_Q\phi^*_Z\ulF^* \simeq \Si{4} H_Q\phi^*_Z\ulF\). For \(j\geq 2\), the homotopy Mackey functors of \(\Si{j\rho_Q} H_Q \phi^*_Z\ulF^*\) are
\begin{align*}
    \ul\pi_i( \Si{j\rho_Q} H_Q\phi^*_Z\ulF^*) &\cong \left\lbrace \begin{array}{ll}
        \phi^*_Z\ulF & i=4j \\
        \ul{mg} & i=4j-1 \\
        \phi^*_{LDR}\ulF\oplus g^{4j-2-i} & i\in[2j+2,4j-2] \\
        \ulg^{2(i-j)-5} & i\in[j+3,2j+1].
    \end{array}\right.
\end{align*}
\end{prop}

Finally, we have the homotopy of \(\Si{j\rho_Q} H_Q\ul w\) and \(\Si{j\rho_Q} H_Q\ul w^*\).

\begin{prop} \label{HomologySkrhoQw}
For \(j\geq 1\), the homotopy Mackey functors of \(\Si{j\rho_Q} H_Q\ul w\) are
\begin{align*}
    \ul\pi_i( \Si{j\rho_Q} H_Q\ul w) &\cong \left\lbrace \begin{array}{ll}
        \phi^*_Z\ulF & i=4j \\
        \ul{mg} & i=4j-1 \\
        \phi^*_{LDR}\ulF\oplus g^{4j-2-i} & i\in[2j,4j-2] \\
        \ulg^{2(i-j)+1} & i\in[j+1,2j-1].
    \end{array}\right.
\end{align*}
\end{prop}

\begin{proof}
The underlying spectrum of \(\Si{j\rho_Q} H_Q\ul w\) is contractible; thus,
\begin{align*}
    \Si{j\rho_Q} H_Q\ul w \simeq \Si{j\rho_K} H_Q\ul w.
\end{align*}
Then, because \(\ul w\) is a pullback over \(Z\), the calculation is essentially $K$-equivariant. Consider the short exact sequence of \(K\)-Mackey functors \(\ul w\rtarr \ulF \rtarr \ulg\) and the corresponding cofiber sequence \(\Si{j\rho_K} H_K\ul w \rtarr \Si{j\rho_K} H_K\ulF \rtarr \Si{j\rho_K} H_K\ulg\). The statement follows immediately from the resulting long exact sequence in homotopy.
\end{proof}

\begin{prop}
\label{HomologySkrhoQwstar}
For \(j=1\), the homotopy Mackey functors of \(\Si{j\rho_Q} H_Q\ul w^*\) are
\begin{align*}
    \ul\pi_i( \Si{j\rho_Q} H_Q\ul w^*) &\cong \left\lbrace \begin{array}{ll}
        \phi^*_Z\ulF & i=4 \\
        \ulg & i=2.
    \end{array}\right.
\end{align*}
For \(j\geq 2\), they are
\begin{align*}
    \ul\pi_i( \Si{j\rho_Q} H_Q\ul w^*) &\cong \left\lbrace \begin{array}{ll}
        \phi^*_Z\ulF & i=4j \\
        \ul{mg} & i=4j-1 \\
        \phi^*_{LDR}\ulF\oplus g^{4j-2-i} & i\in[2j+2,4j-2] \\
        \ulg^{2(i-j)-5} & i\in[j+3,2j+1] \\
        \ulg & i=j+1.
    \end{array}\right.
\end{align*}
\end{prop}

\begin{proof}
The proof is the same as that in \cref{HomologySkrhoQw}, except that we start with the exact sequence of \(K\)-Mackey functors \(\ulg \rtarr \ulF^* \rtarr \ul w^*\).
\end{proof}

%%%%%%%%%%%%%%%%%%%%%%%%%%%%%%%
\section{Slice spectral sequences}
\label{sec:SpectralSequences}

Here we include the slice spectral sequences for \(\Si{n} H_Q\ulZ\) for several values of \(n\) between 5 and 15. 
In some cases, we use the restriction to the $C_4$-subgroups to determine some of the slice differentials.

The grading is the same as that in \cite[Section 4.4.2]{Kervaire}. The Mackey functor \(\ul E^{t-n,t}_2\) is \(\ul\pi_nP^t_t(X)\). We also follow the Adams convention, where \(\ul\pi_nP^t_t(X)\) has coordinates \((n,t-n)\) and the differential
\begin{align*}
    d_r: \ul E^{s,t}_r \rtarr \ul E^{s+r,t+r-1}_r
\end{align*}
points left one and up \(r\).

The $Q$-Mackey functors that appear in these spectral sequences are listed in \cref{SSS:Qmack}. We also display some companion $C_4$-slice spectral sequences, and the $C_4$-Mackey functors that appear are listed in \cref{SSS:C4mack}.

\begin{table}[h] 
\caption[Symbols for {$Q$}-Mackey functors]{Symbols for $Q$-Mackey functors}
\label{SSS:Qmack}
\begin{center}
	{\renewcommand{\arraystretch}{1.5}
		\begin{tabular}{|c|c|c|}
			\hline 
			$\square=\ulZ$
			& $\raisebox{-0.3ex}{$\phiZF$} = \phi^*_Z\ulF$
			& $\bpent =\phi^*_{LDR} \ulF$
			\\\hline
			$\mysterysymbol = \mystery$
			& $\circ = \ul B(3,0)$
			& $\phiZM = \phi^*_Z\ul B(2,0)$
			\\ \hline
			 $\btrap =\ul{mg}$
			 & $\gcirc=\ulg^n$ &  \\ \hline 
	\end{tabular} }
\end{center}
\end{table}

\begin{table}[h] 
\caption[Symbols for {$C_4$}-Mackey functors]{Symbols for $C_4$-Mackey functors}
\label{SSS:C4mack}
\begin{center}
	{\renewcommand{\arraystretch}{1.5}
		\begin{tabular}{|c|c|c|c|}
			\hline 
			$\square=\ulZ$
			& $\raisebox{-0.3ex}{$\phiZF$} = \phi^*_{C_2}\ulF$
			& $\circ = \ul B(2,0)$
			 & $\bullet=\ulg$  \\ \hline 
	\end{tabular} }
\end{center}
\end{table}

\begin{eg}
In the spectral sequences for \(\Si{5} H_Q\ulZ\), \(\Si{6} H_Q\ulZ\), and \(\Si{7} H_Q\ulZ\), because we must be left with
\begin{align*}
    \ul\pi_n( P^n_n\Si{n} H_Q\ulZ) \cong \ulZ,
\end{align*}
all differentials are forced.
\end{eg}

\begin{eg} \label{SSS8Z}
For \(\Si{8} H_Q\ulZ\), the pattern of differentials emanating from the Mackey functor $\mpi_6 (P^8_8 \Si8 H_Q \ulZ)$ is forced; no other pattern of differentials wipes out all classes in this region.
The shorter differentials clearing out the smaller region are then similarly forced. 
\end{eg}

\begin{eg}
In the cases of $\Si{n}H_Q\ulZ$ for $n=10$, $12$, and 15, we also display the corresponding slice spectral sequence for $\Si{n} H_{C_4}\ulZ$, where we use $C_4$ to indiscriminately refer to any of the subgroups $L, D, R \leq Q$. The slice differentials in the $C_4$-case force many of the slice differentials for the $Q$-equivariant spectra.
\end{eg}

\vfill

%%%%% The spectral sequences are commented out
%%%%%  to speed up compilation time. Now they are 
%%%%%%  just being imported as precompiled pdf's

%%%%%%%%%%%%%%%%%%%%%%%%%%%%%%%%%%%%%%%%%%%%%%%%%%%%%%%%%%%%%%%%%%
%%%%%%%%%%%%%%%%%%%%%%%%%%%%%%%%%%%%%%%%%%%%%%%%%%%%%%%%%%%%%%%%%%

\includegraphics{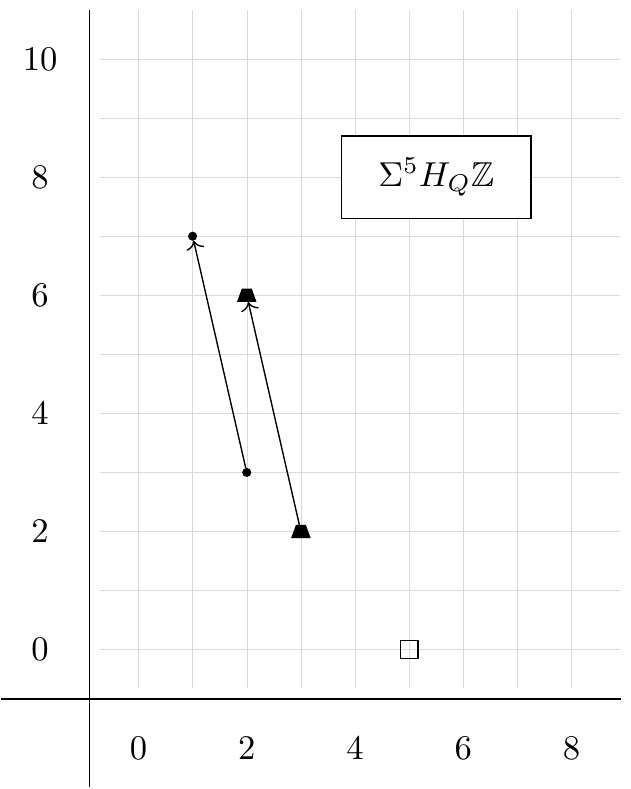}
\includegraphics{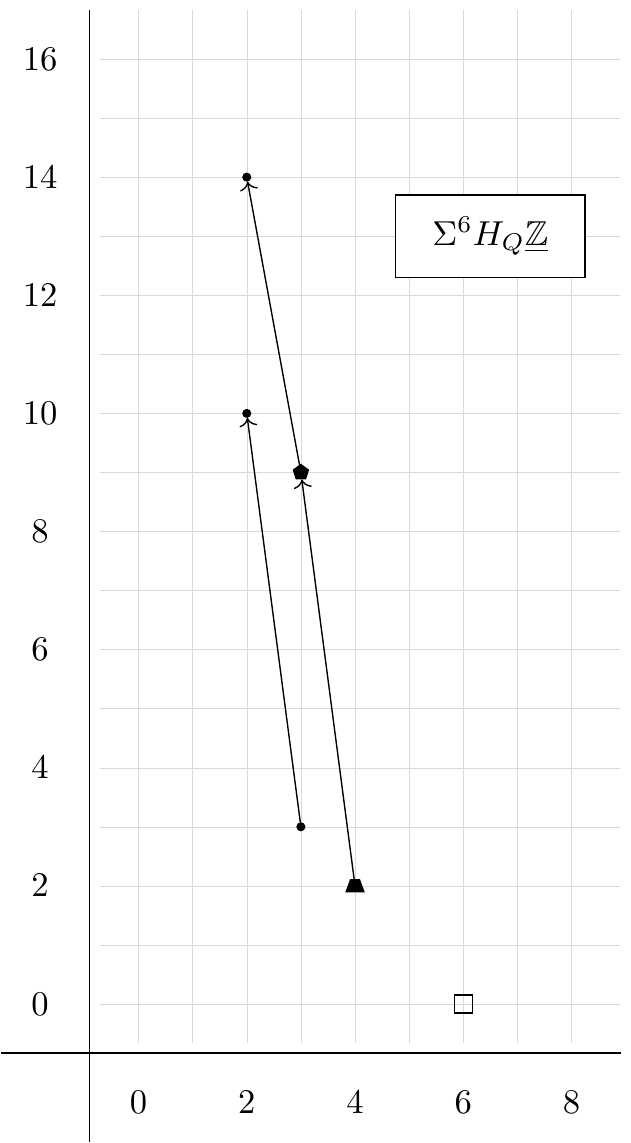}

\includegraphics[width=0.475\textwidth]{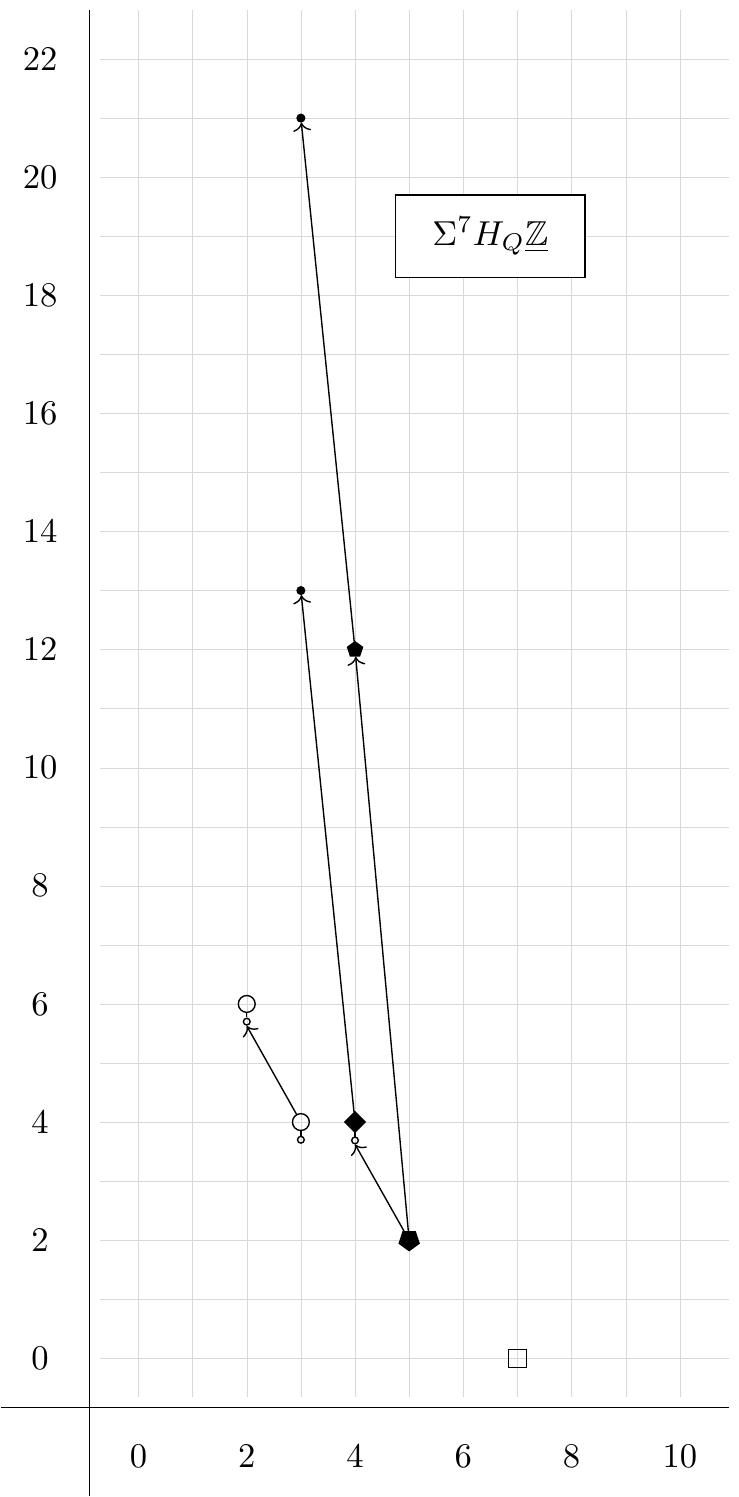}
\includegraphics[width=0.475\textwidth]{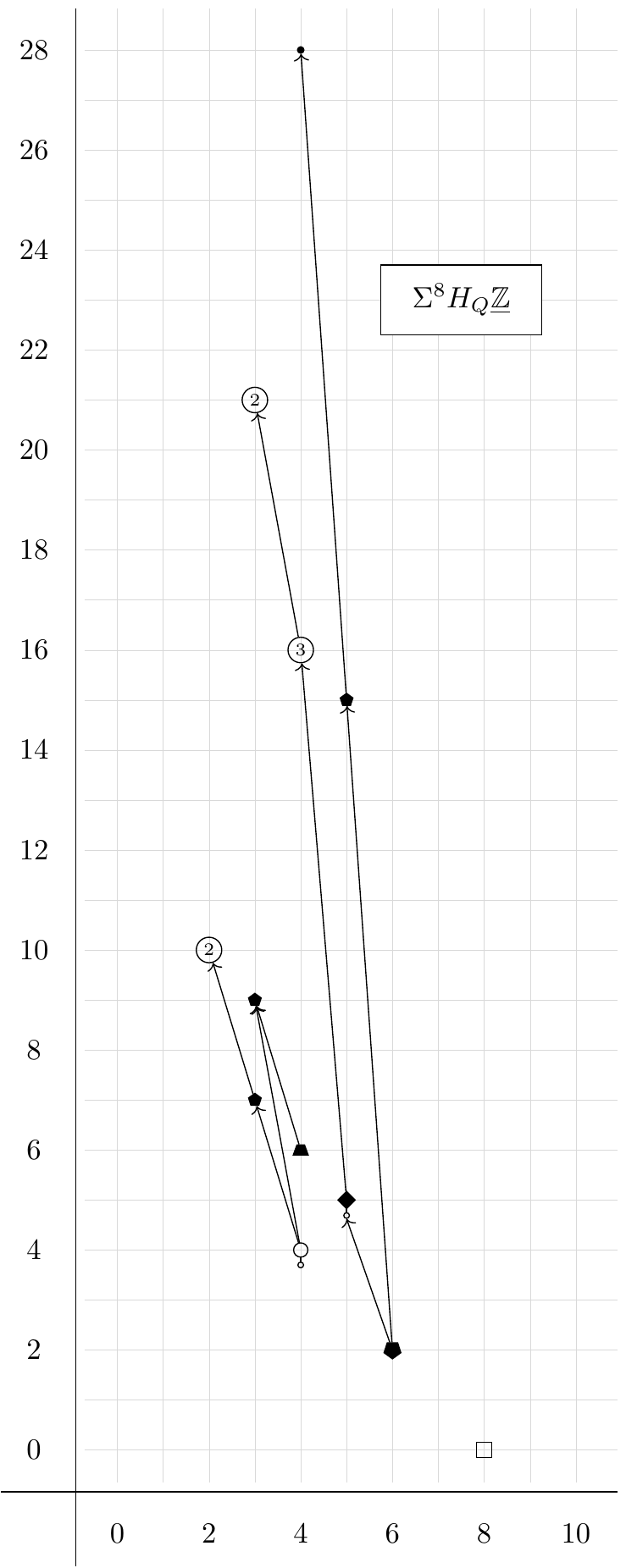}

\includegraphics{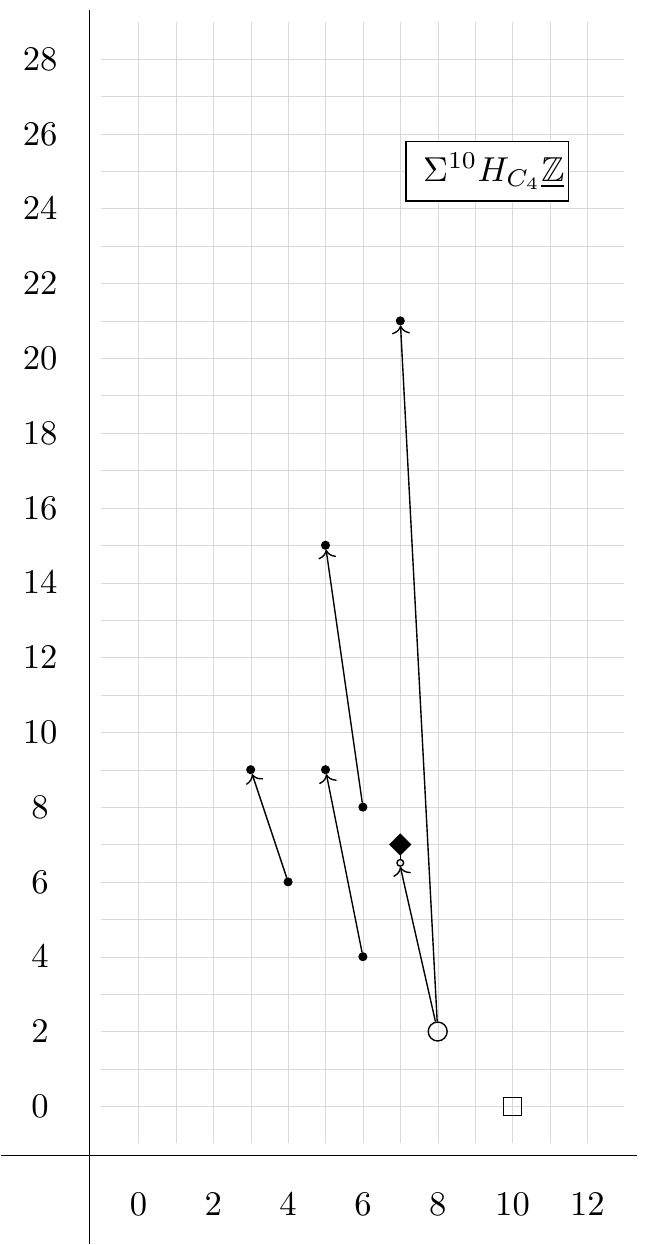}
\includegraphics{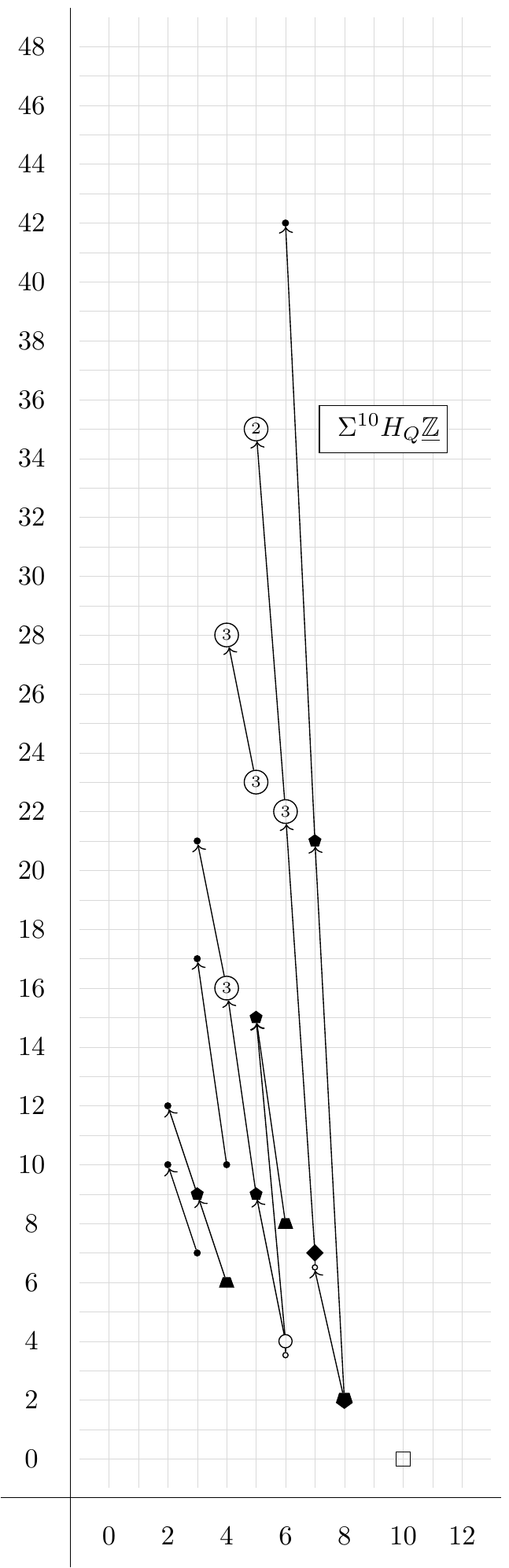}

\includegraphics{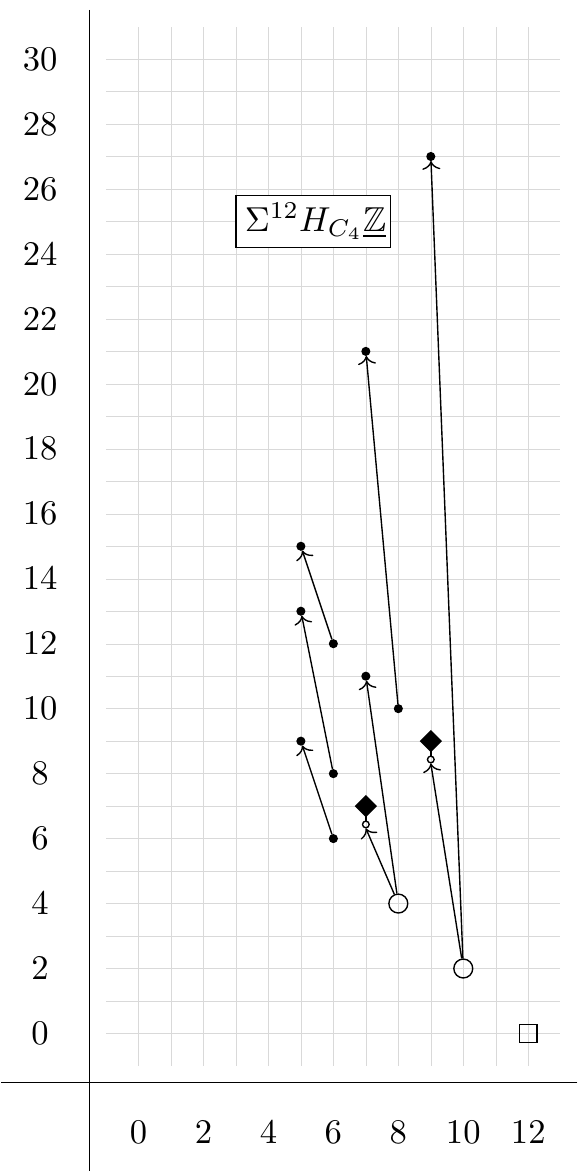}
\includegraphics{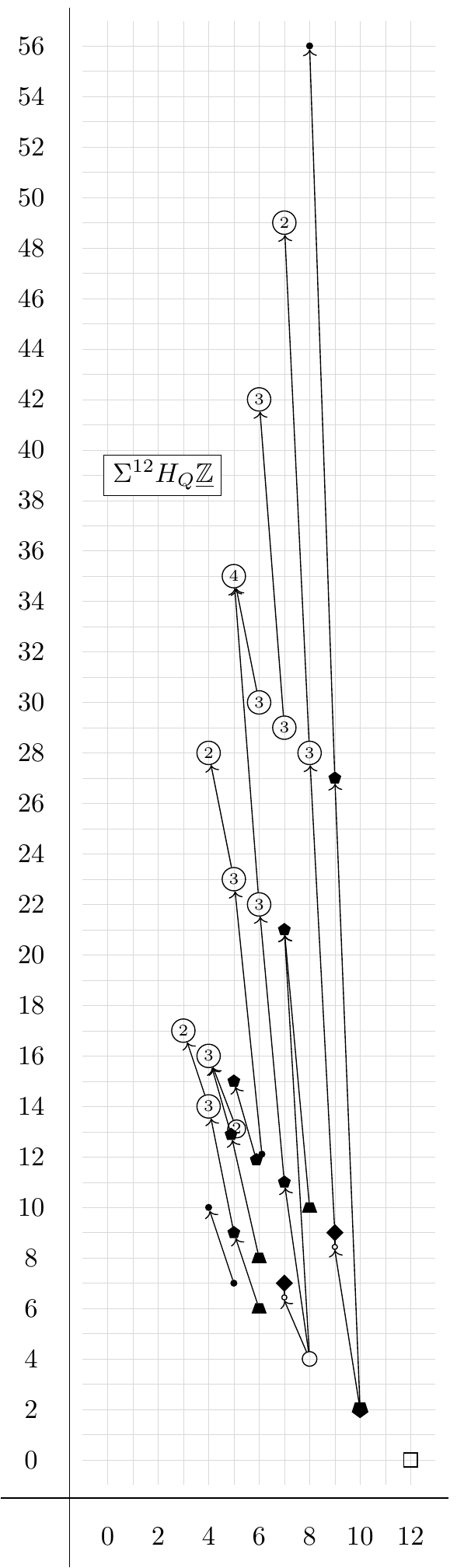}

\includegraphics[width=0.475\textwidth]{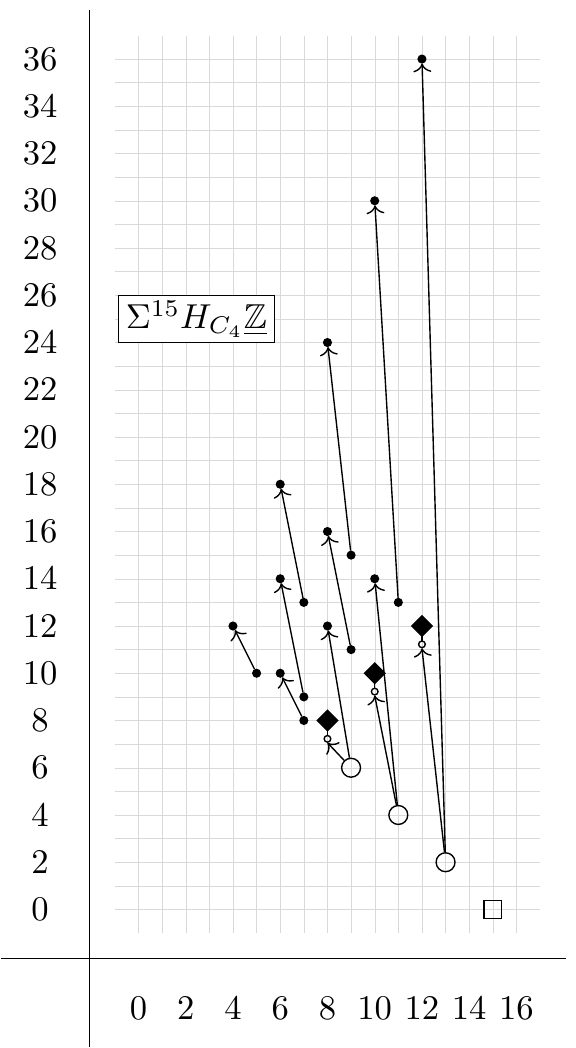}
\includegraphics[width=0.475\textwidth]{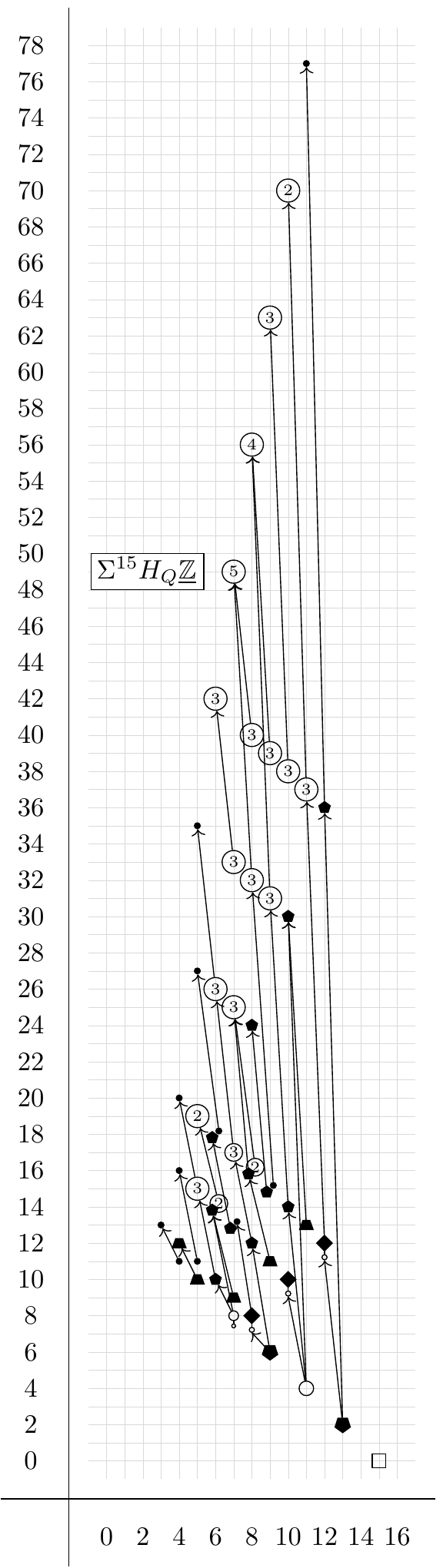}

\bibliographystyle{amsalpha}

\begin{bibdiv}
\begin{biblist}

\bib{V}{article}{
  eprint = {https://arxiv.org/abs/2111.06559},  
  author = {Angeltveit, Vigleik},
  title = {The slice spectral sequence for the cyclic group of order p},
  year = {2021},
}


\bib{BDS}{article}{
   author={Balmer, Paul},
   author={Dell'Ambrogio, Ivo},
   author={Sanders, Beren},
   title={Grothendieck-Neeman duality and the Wirthm\"{u}ller isomorphism},
   journal={Compos. Math.},
   volume={152},
   date={2016},
   number={8},
   pages={1740--1776},
   issn={0010-437X},
   review={\MR{3542492}},
   doi={10.1112/S0010437X16007375},
}
	

\bib{BasuGhosh}{article}{
    author = {Basu, S.},
    author = {Ghosh, S.},
    title = {Non-trivial extensions in equivariant cohomology with constant coefficients},
      year={2021},
      eprint={https://arxiv.org/abs/2108.12763},
}
    
\bib{BSW}{article}{
   author={Bouc, Serge},
   author={Stancu, Radu},
   author={Webb, Peter},
   title={On the projective dimensions of Mackey functors},
   journal={Algebr. Represent. Theory},
   volume={20},
   date={2017},
   number={6},
   pages={1467--1481},
   issn={1386-923X},
   review={\MR{3735915}},
   doi={10.1007/s10468-017-9695-y},
}

\bib{Brown}{book}{
   author={Brown, Kenneth S.},
   title={Cohomology of groups},
   series={Graduate Texts in Mathematics},
   volume={87},
   publisher={Springer-Verlag, New York-Berlin},
   date={1982},
   pages={x+306},
   isbn={0-387-90688-6},
   review={\MR{672956}},
}

\bib{Dug}{article}{
   author={Dugger, Daniel},
   title={An Atiyah-Hirzebruch spectral sequence for $KR$-theory},
   journal={$K$-Theory},
   volume={35},
   date={2005},
   number={3-4},
   pages={213--256 (2006)},
   issn={0920-3036},
   review={\MR{2240234}},
   doi={10.1007/s10977-005-1552-9},
}

\bib{GY}{article}{
   author={Guillou, B.},
   author={Yarnall, C.},
   title={The Klein four slices of $\Sigma^nH\underline{\mathbb F}_2$},
   journal={Math. Z.},
   volume={295},
   date={2020},
   number={3-4},
   pages={1405--1441},
   issn={0025-5874},
   review={\MR{4125695}},
   doi={10.1007/s00209-019-02433-3},
}

\bib{SlicePrimer}{article}{
   author={Hill, Michael A.},
   title={The equivariant slice filtration: a primer},
   journal={Homology Homotopy Appl.},
   volume={14},
   date={2012},
   number={2},
   pages={143--166},
   issn={1532-0073},
   review={\MR{3007090}},
   doi={10.4310/HHA.2012.v14.n2.a9},
}

\bib{Kervaire}{article}{
   author={Hill, M. A.},
   author={Hopkins, M. J.},
   author={Ravenel, D. C.},
   title={On the nonexistence of elements of Kervaire invariant one},
   journal={Ann. of Math. (2)},
   volume={184},
   date={2016},
   number={1},
   pages={1--262},
   issn={0003-486X},
   review={\MR{3505179}},
   doi={10.4007/annals.2016.184.1.1},
}
		
\bib{HHR}{article}{
   author={Hill, Michael A.},
   author={Hopkins, Michael J.},
   author={Ravenel, Douglas C.},
   title={The slice spectral sequence for the $C_4$ analog of real
   $K$-theory},
   journal={Forum Math.},
   volume={29},
   date={2017},
   number={2},
   pages={383--447},
   issn={0933-7741},
   review={\MR{3619120}},
   doi={10.1515/forum-2016-0017},
}

\bib{HHR2}{article}{
   author={Hill, Michael A.},
   author={Hopkins, M. J.},
   author={Ravenel, D. C.},
   title={The slice spectral sequence for certain $RO(C_{p^n})$-graded
   suspensions of $H\underline{\bf Z}$},
   journal={Bol. Soc. Mat. Mex. (3)},
   volume={23},
   date={2017},
   number={1},
   pages={289--317},
   issn={1405-213X},
   review={\MR{3633137}},
   doi={10.1007/s40590-016-0129-3},
}

\bib{HSWX}{article}{
  author = {Hill, Michael A.},
  author = {Shi, XiaoLin Danny},
  author = {Wang, Guozhen},
  author = {Xu, Zhouli},  
  title = {The slice spectral sequence of a $C_4$-equivariant height-4 Lubin-Tate theory},
  year = {2018},
  eprint = {https://arxiv.org/abs/1811.07960},
}


\bib{HKBPO}{article}{
   author={Hu, Po},
   author={Kriz, Igor},
   title={The homology of $BPO$},
   conference={
      title={Recent progress in homotopy theory},
      address={Baltimore, MD},
      date={2000},
   },
   book={
      series={Contemp. Math.},
      volume={293},
      publisher={Amer. Math. Soc., Providence, RI},
   },
   date={2002},
   pages={111--123},
   review={\MR{1887531}},
   doi={10.1090/conm/293/04945},
}
		

\bib{SKriz}{article}{
   author={Kriz, Sophie},
   title={Notes on equivariant homology with constant coefficients},
   journal={Pacific J. Math.},
   volume={309},
   date={2020},
   number={2},
   pages={381--399},
   issn={0030-8730},
   review={\MR{4202017}},
   doi={10.2140/pjm.2020.309.381},
}



\bib{LMS}{book}{
   author={Lewis, L. G., Jr.},
   author={May, J. P.},
   author={Steinberger, M.},
   author={McClure, J. E.},
   title={Equivariant stable homotopy theory},
   series={Lecture Notes in Mathematics},
   volume={1213},
   note={With contributions by J. E. McClure},
   publisher={Springer-Verlag, Berlin},
   date={1986},
   pages={x+538},
   isbn={3-540-16820-6},
   review={\MR{866482}},
   doi={10.1007/BFb0075778},
}
	
\bib{Lu}{article}{
    title={On the $RO(G)$-graded coefficients of $Q_8$ equivariant cohomology}, 
    author={Lu, Yunze},
    year={2021},
    eprint={https://arxiv.org/abs/2111.01926},
}

\bib{MSZ}{article}{
  author = {Meier, Lennart},
  author = {Shi, XiaoLin Danny},
  author = {Zeng, Mingcong},
  title = {The localized slice spectral sequence, norms of Real bordism, and the Segal conjecture},
  year = {2020},
  eprint = {https://arxiv.org/abs/2008.04963},
}


\bib{Slone}{article}{
      title={Klein Four 2-slices and the Slices of $\Sigma^{\pm n}H\underline{\mathbb{Z}}$}, 
      author={Slone, Carissa},
      year={2021},
      eprint={https://arxiv.org/abs/2106.02767},
      note={To appear in {\it Math. Z}},
}

\bib{Sulyma}{book}{
   author={Sulyma, Yuri John Fraser},
   title={Equivariant Aspects of Topological Hochschild Homology},
   note={Thesis (Ph.D.)--The University of Texas at Austin},
   publisher={ProQuest LLC, Ann Arbor, MI},
   date={2019},
   pages={95},
   isbn={979-8678-19586-9},
   review={\MR{4197745}},
}

\bib{TW}{article}{
   author={Th\'{e}venaz, Jacques},
   author={Webb, Peter},
   title={The structure of Mackey functors},
   journal={Trans. Amer. Math. Soc.},
   volume={347},
   date={1995},
   number={6},
   pages={1865--1961},
   issn={0002-9947},
   review={\MR{1261590}},
   doi={10.2307/2154915},
}



\bib{Ullman}{article}{
   author={Ullman, John},
   title={On the slice spectral sequence},
   journal={Algebr. Geom. Topol.},
   volume={13},
   date={2013},
   number={3},
   pages={1743--1755},
   issn={1472-2747},
   review={\MR{3071141}},
   doi={10.2140/agt.2013.13.1743},
}
		

\bib{Waner}{article}{
  author={Waner, Stefan},
  title={Periodicity in the cohomology of universal {$G$}-spaces},
   journal={Illinois J. Math.},
   volume={30},
   date={1986},
   number={3},
   pages={468--478},
   issn={0019-2082},
   review={\MR{850344}},
   }

\bib{Yarn}{book}{
   author={Yarnall, Carolyn Marie},
   title={The Slices of $S^{n \lambda} HZ$ for Cyclic p-Groups},
   note={Thesis (Ph.D.)--University of Virginia},
   publisher={ProQuest LLC, Ann Arbor, MI},
   date={2013},
   pages={99},
   isbn={978-1303-39719-6},
   review={\MR{3187569}},
}

\bib{Yarnall}{article}{
   author={Yarnall, Carolyn},
   title={The slices of $S^n\wedge H\underline{\Z}$ for cyclic
   $p$-groups},
   journal={Homology Homotopy Appl.},
   volume={19},
   date={2017},
   number={1},
   pages={1--22},
   issn={1532-0073},
   review={\MR{3628673}},
   doi={10.4310/HHA.2017.v19.n1.a1},
}

\bib{Z}{article}{
    author={Zeng, M.},
    title={Equivariant Eilenberg-Mac~Lane spectra in cyclic $p$-groups},
    year={2018},
    eprint={https://arxiv.org/abs/1710.01769},
}

\bib{Zou}{book}{
    author={Zou, Y.},
    title={$RO(D_{2p})$-graded slice spectral sequence for $HZ$},
    year={2018},
    note={Thesis (Ph.D.)-University of Rochester},
}
		

\end{biblist}
\end{bibdiv}

\end{document}